\documentclass[a4paper,11pt]{article}

\usepackage[T1]{fontenc}
\usepackage{lmodern}

\usepackage{dsfont}

\usepackage[latin1]{inputenc}
\usepackage{amsmath}
\usepackage{amsthm}
\usepackage{amssymb}
\usepackage{mathrsfs}
\usepackage{graphicx}
\usepackage[all]{xy}
\usepackage{hyperref}

\usepackage{makeidx}

\usepackage{stmaryrd}
\usepackage{caption}

\usepackage{abstract} 

\newtheorem{thm}{Theorem}[section]
\newtheorem{cor}[thm]{Corollary}
\newtheorem{claim}[thm]{Claim}

\newtheorem{lemma}[thm]{Lemma}
\newtheorem{prop}[thm]{Proposition}

\theoremstyle{definition}
\newtheorem{definition}[thm]{Definition}

\def\rquotient#1#2{%
	\makeatletter
	\raise.3ex\hbox{$#1$}/\lower.3ex\hbox{$#2$}%
	\makeatother
}	

\makeatletter
\newcommand{\subjclass}[2][2010]{%
	\let\@oldtitle\@title%
	\gdef\@title{\@oldtitle\footnotetext{#1 \emph{Mathematics subject classification.} #2}}%
}
\newcommand{\keywords}[1]{%
	\let\@@oldtitle\@title%
	\gdef\@title{\@@oldtitle\footnotetext{\emph{Key words and phrases.} #1.}}%
}
\makeatother

\newcommand{\Address}{{
		\bigskip
		\small
		
		\textsc{University of Montpellier\\ 
Institut Math\'ematiques Alexander Grothendieck\\
Place Eug\`ene Bataillon\\
34090 Montpellier (France)}\par\nopagebreak
		\textit{E-mail address}: \texttt{anthony.genevois@umontpellier.fr}
		
}}

\makeindex

\title{Rotation groups virtually embed into right-angled rotation groups}
\date{\today}
\author{Anthony Genevois}


\begin{document}

\maketitle

\begin{abstract}
It is a theorem due to F. Haglund and D. Wise that reflection groups (aka Coxeter groups) virtually embed into right-angled reflection groups (aka right-angled Coxeter groups). In this article, we generalise this observation to \emph{rotation groups}, which can be thought of as a common generalisation of Coxeter groups and graph products of groups. More precisely, we prove that rotation groups (aka \emph{periagroups}) virtually embed into right-angled rotation groups (aka graph products of groups). 
\end{abstract}

\tableofcontents

\section{Introduction}

\noindent
It is well-known that, in some sense, Coxeter groups are precisely the groups that arrive as abstract reflection groups. More precisely, a \emph{reflection system} is the data of a group $G$, of a subset $R \subset G$, and a connected graph $\Omega$ on which $G$ acts such that:
\begin{itemize}
	\item $R$ is stable under conjugation and it generates $G$;
	\item every edge of $\Omega$ is flipped by a unique element of $R$ and every element of $R$ flips an edge of $\Omega$;
	\item for every $r \in R$, removing all the edges flipped by $r$ separates the endpoints of any edge flipped by $r$.
\end{itemize}
Then a group is part of a reflection system if and only if it is a Coxeter group. In \cite{Mediangle}, the author proposed a natural generalisation of reflection systems as \emph{rotation systems}. The basic idea is that one replaces involutions flipping endpoints of edges with arbitrary groups acting freely and transitively on complete graphs. More formally, a \emph{rotation system} is the data of a group $G$, a collection of subgroups $\mathcal{R}$, and a connected graph $\Omega$ on which $G$ acts such that:
\begin{itemize}
	\item $\mathcal{R}$ is stable under conjugation and it generates $G$.
	\item For every clique $C \subset \Omega$, there exists a unique $R \in \mathcal{R}$ such that $R$ acts freely-transitively on the vertices of $C$. One refers to $R$ as the \emph{rotative-stabiliser} of~$C$.
	\item Every subgroup in $\mathcal{R}$ is the rotative-stabiliser of a clique of $\Omega$.
	\item For every clique $C \subset \Omega$, removing the edges of all the cliques having the same rotative-stabiliser as $C$ separates any two vertices in $C$.
\end{itemize}
Then a group is part of a rotation system if and only if it is a \emph{periagroup} (whose factors come from $\mathcal{R}$). We refer to \cite{Mediangle} for a precise statement regarding the equivalence between rotation groups and periagroups.

\begin{definition}\label{def:Periagroups}
Let $\Gamma$ be a (simplicial) graph, $\lambda : E(\Gamma) \to \mathbb{N}_{\geq 2}$ a labelling of its edges, and $\mathcal{G}= \{G_u \mid u \in V(\Gamma)\}$ a collection of groups indexed by the vertices of $\Gamma$, referred to as the \emph{vertex-groups}. Assume that, for every edge $\{u,v\} \in E(\Gamma)$, if $\lambda(\{u,v\})>2$ then $G_u,G_v$ have order two. The \emph{periagroup} $\Pi(\Gamma,\lambda,\mathcal{G})$ admits
$$\left\langle G_u \ (u \in V(\Gamma)) \mid \langle G_u,G_v \rangle^{\lambda(u,v)}= \langle G_v,G_u \rangle^{\lambda(u,v)} \ (\{u,v\} \in E(\Gamma)) \right\rangle$$
as a relative presentation. Here, $\langle a,b \rangle^k$ refers to the word obtained from $ababab \cdots$ by keeping only the first $k$ letters; and $\langle G_u,G_v \rangle^k = \langle G_v,G_u \rangle^k$ is a shorthand for: $\langle a, b \rangle^k = \langle b,a \rangle^k$ for all non-trivial $a \in G_u$, $b \in G_v$. 
\end{definition}

\noindent
Periagroups of cyclic groups of order two coincide with Coxeter groups; and, if $\lambda \equiv 2$, all the relations are commutations and one retrieves graph products of groups (such as right-angled Artin groups). Thus, periagroups can be thought of as an interpolation between Coxeter groups and graph products of groups. Periagroups of cyclic groups, a common generalisation of Coxeter groups and right-angled Artin groups, also appear in \cite{MR1076077}, and are studied geometrically in \cite{Soergel} under the name \emph{Dyer groups}. Periagroups of cyclic groups of order two or infinite can also be found in \cite{MR1930963} as \emph{weakly partially commutative Artin-Coxeter groups}. 

\medskip \noindent
It is proved in \cite{MR2646113} that (finitely generated) Coxeter groups virtually embeds in right-angled Coxeter groups. In other words, a reflection group always virtually embeds into some right-angled reflection group. In this article, we extend this assertion from reflection groups to arbitrary rotation groups. More formally:

\begin{thm}\label{thm:BigIntro}
Let $\Pi(\Gamma, \mathcal{G},\lambda)$ be a periagroup with $\Gamma$ finite. There exists a finite graph $\Phi$ and a collection $\mathcal{H}$ indexed by $V(\Phi)$ of groups isomorphic to groups from $\mathcal{G}$ such that the periagroup $\Pi(\Gamma, \mathcal{G},\lambda)$ virtually embeds into the graph product $\Phi \mathcal{H}$ as a virtual retract. 
\end{thm}

\noindent
In addition to being conceptually appealing, Theorem~\ref{thm:BigIntro} also allows us to deduce various combination results, since many properties stable under taking subgroups and under commensurability are known to hold for graph products of groups. For instance:

\begin{cor}
Let $\Pi:=\Pi(\Gamma, \mathcal{G},\lambda)$ be a periagroup with $\Gamma$ finite. 
\begin{itemize}
	\item If the groups in $\mathcal{G}$ are all linear, then $\Pi$ is linear. \cite{MR4422610, GreenGP}
	\item If the groups in $\mathcal{G}$ are all a-T-menable (resp.\ weakly amenable), then $\Pi$ is a-T-menable (resp.\ weakly amenable). \cite{HaagerupGP, QM, MR3687943}
	\item For every $n \geq 1$, if the groups in $\mathcal{G}$ are all of type $F_n$, then $\Pi$ is of type $F_n$. \cite{MR1293049, MR1317337, MR1377652}
\end{itemize}
\end{cor}

\noindent
The proof of Theorem~\ref{thm:BigIntro} is geometric. It essentially follows Haglund and Wise's strategy, but in the more general framework of \emph{mediangle} and \emph{quasi-median} graphs. Given our periagroup $\Pi:= \Pi(\Gamma, \mathcal{G},\lambda)$, we consider its Cayley graph $M:=M(\Gamma, \mathcal{G},\lambda)$ associated to the (usually finite) generating set obtained by taking the union of all the vertex-groups. As shown in \cite{Mediangle}, this graph admits a nice geometric structure: it is a \emph{mediangle graph}. (See Sections~\ref{section:Mediangle} and~\ref{section:Periagroups} for more details.) In particular, it has a natural hyperplane structure that allows us, following \cite{QM}, to construct a canonical \emph{quasi-median graph} $\mathrm{QM}:= \mathrm{QM}(\Gamma, \mathcal{G},\lambda)$. More precisely, there is an isometric embedding $M \hookrightarrow \mathrm{QM}$ and the action of $\Pi$ on $M$ extends to $\mathrm{QM}$. (See Sections~\ref{section:QM} and~\ref{section:Popset} for more details.) Following \cite{MR4586831}, it suffices to verify that the action of $\Pi$ on $\mathrm{QM}$ is \emph{conspicial} in order to conclude that $\Pi$ embeds into some graph products. (See Section~\ref{section:Conspicial} for more details.) Unfortunately, the action is usually not conspicial. Our key observation is that the obstruction to conspiciality lies in a specific part of the periagroup that can be avoided by a finite-index subgroup. See Proposition~\ref{prop:Obstruction} for a precise statement. Then, bringing all the pieces together, Theorem~\ref{thm:BigIntro} follows.

\section{Preliminaries}

\subsection{Quasi-median graphs}\label{section:QM}

\noindent
There exist several equivalent definitions of quasi-median graphs; see for instance \cite{quasimedian}. For instance, quasi-median graphs can be defined as retracts of Hamming graphs (i.e.\ products of complete graphs). However, this definition, despite being conceptually simple, is difficult to use in practice. Below is the definition we used in \cite{QM}. It sounds more technical, but it is easier to use in practice, and it will be relevant regarding our definition of mediangle graphs in Section~\ref{section:Mediangle}. 

\begin{minipage}{0.4\linewidth}
\noindent
\includegraphics[scale=0.35]{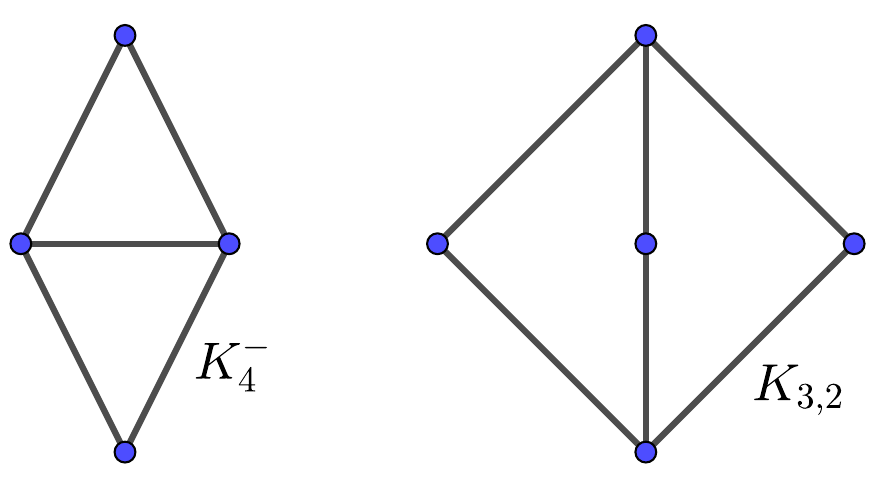}
\end{minipage}
\begin{minipage}{0.59\linewidth}
Recall that the graph $K_{3,2}$ is the bipartite complete graph, corresponding to two squares glued along two adjacent edges; and $K_4^-$ is the complete graph on four vertices minus an edge, corresponding to two triangles glued along an edge. 
\end{minipage}

\begin{definition}
A connected graph $X$ is \emph{weakly modular} if it satisfies the following two conditions:
\begin{description}
	\item[(triangle condition)] for every vertices $a, x,y \in X$, if $x$ and $y$ are adjacent and if $d(a,x)=d(a,y)$, then there exists a vertex $z \in X$ which adjacent to both $x$ and $y$ and which satisfies $d(a,z)=d(a,x)-1$;
	\item[(quadrangle condition)] for every vertices $a,x,y,z \in X$, if $z$ is adjacent to both $x$ and $y$ and if $d(a,x)=d(a,y)=d(a,z)-1$, then there exists a vertex $w \in X$ which adjacent to both $x$ and $y$ and which satisfies $d(a,w)=d(a,z)-2$.
\end{description}
If moreover $X$ does not contain $K_4^-$ and $K_{3,2}$ as induced subgraphs, then it is \emph{quasi-median}. 
\end{definition}

\noindent
The triangle and quadrangle conditions are illustrated by Figure \ref{Quadrangle}.
\begin{figure}[h!]
\begin{center}
\includegraphics[scale=0.45]{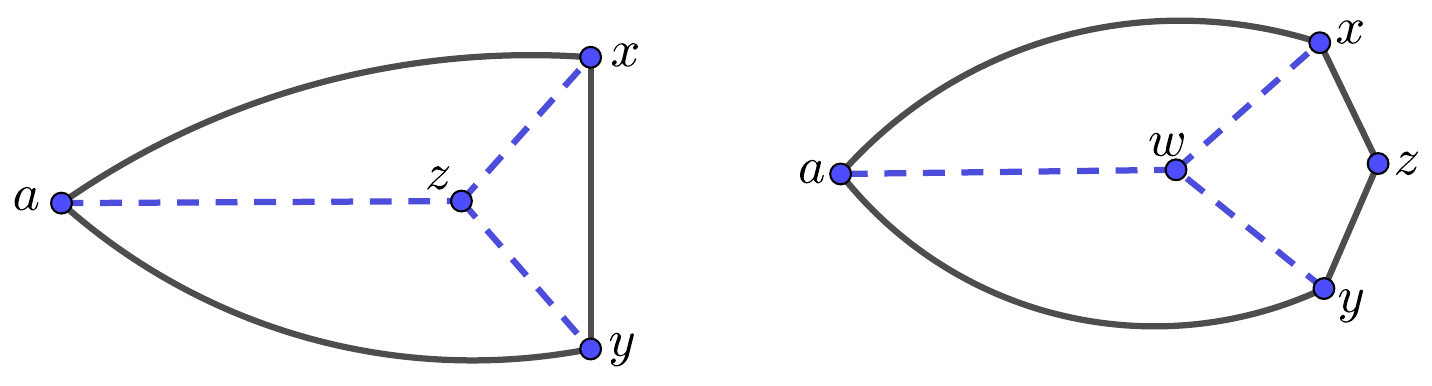}
\caption{Triangle and quadrangle conditions.}
\label{Quadrangle}
\end{center}
\end{figure}

\medskip \noindent
Examples of quasi-median graphs include median graphs, also known as one-skeletons of CAT(0) cube complexes, and products of complete graphs.

\medskip \noindent
Recall that a \emph{clique} is a maximal complete subgraph. According to \cite{quasimedian}, cliques in quasi-median graphs satisfy a strong convexity condition; namely, they are \emph{gated} in the following sense. 

\begin{definition}
Let $X$ be a graph and $Y \subset X$ a subgraph. A vertex $y \in Y$ is a \emph{gate} of an other vertex $x \in X$ if, for every $z \in Y$, there exists a geodesic between $x$ and $z$ passing through $y$. If every vertex of $X$ admits a gate in $Y$, then $Y$ is \emph{gated}.
\end{definition}

\noindent
Given a gated subgraph $Y$, we will refer to the map that associates to every vertex its gate in $Y$ as the \emph{projection to $Y$}. Accordingly, the gate of vertex may be refer to as its projection.

\medskip \noindent
Like for median graphs, a fundamental tool in the study of quasi-median graphs is given by \emph{hyperplanes}. 

\begin{definition}
Let $X$ be a graph. A \emph{hyperplane} $J$ is an equivalence class of edges with respect to the transitive closure of the relation saying that two edges are equivalent whenever they belong to a common triangle or are opposite sides of a square. We denote by $X \backslash \backslash J$ the graph obtained from $X$ by removing the interiors of all the edges of $J$. A connected component of $X \backslash \backslash J$ is a \emph{sector}. The \emph{carrier} of $J$, denoted by $N(J)$, is the subgraph generated by all the edges of $J$. Two hyperplanes $J_1$ and $J_2$ are \emph{transverse} if there exist two edges $e_1 \subset J_1$ and $e_2 \subset J_2$ spanning a $4$-cycle in $X$; and they are \emph{tangent} if they are not transverse but $N(J_1) \cap N(J_2) \neq \emptyset$. 
\end{definition}
\begin{figure}
\begin{center}
\includegraphics[trim={0 16.5cm 10cm 0},clip,scale=0.45]{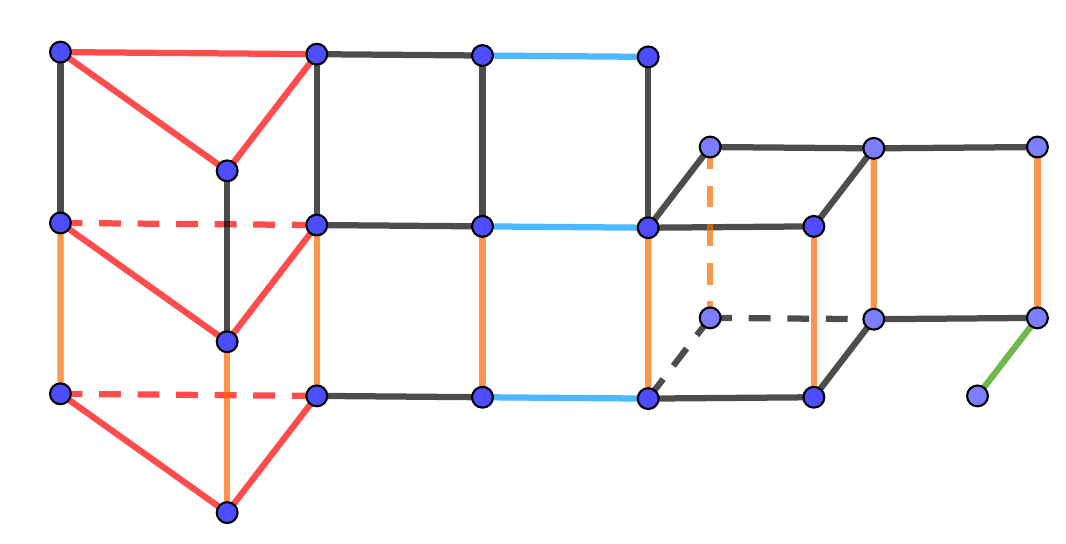}
\caption{A quasi-median graph and some of its hyperplanes. The orange hyperplane is transverse to the red and blue hyperplanes. The green and orange hyperplanes are tangent. The red and blue hyperplanes are neither transverse nor tangent.}
\label{figure3}
\end{center}
\end{figure}

\noindent
See Figure \ref{figure3} for examples of hyperplanes in a quasi-median graph. The key point is that, in some sense, the geometry of a quasi-median graph reduces to the combinatorics of its hyperplanes. This is illustrated by Theorem~\ref{thm:BigHyp} in the more general framework of median graphs.

\subsection{Spaces with partitions}\label{section:Popset}

\noindent
Generalising the construction of a median graph (or equivalently, a CAT(0) cube complex) from a \emph{space with walls}, in \cite{QM} we introduced \emph{spaces with partitions} and showed to construct quasi-median graphs from them. This Section is dedicated to the description of this construction.

\medskip \noindent
Let $X$ be a set and $\mathfrak{P}$ a collection of partitions. We allow two distinct elements of $\mathfrak{P}$ to induce the same partition of $X$, in which case we say that the two elements are \emph{indistinguishable}. We refer to the pieces of the partitions from $\mathfrak{P}$ as \emph{sectors}. 
\begin{itemize}
	\item Two partitions $\mathcal{P},\mathcal{Q}$ are \emph{nested} if there exist sectors $A \in \mathcal{P}$ and $B \in \mathcal{Q}$ such that $D \subset A$ for every $D \in \mathcal{Q} \backslash \{B\}$ and $D \subset B$ for every $D \in \mathcal{P} \backslash \{A\}$.
	\item Two partitions $\mathcal{P},\mathcal{Q}$ are \emph{transverse} if $P$ and $Q$ are not $\subset$-comparable for all $P \in \mathcal{P}$ and $Q\in \mathcal{Q}$.
	\item Two partitions $\mathcal{P},\mathcal{Q}$ are \emph{tangent} if there do not exist $P \in \mathcal{P}$, $Q\in \mathcal{Q}$, $\mathcal{R} \in \mathfrak{P}$, and $R \in \mathcal{R}$ such that $P \subsetneq R \subsetneq Q$ or $Q \subsetneq R \subsetneq P$.
	\item A partition $\mathcal{P}$ separates two points $x,y \in X$ if $x$ and $y$ belong to two distinct sectors of $\mathcal{P}$.
\end{itemize}
The typical example to keep in mind is when $X$ is a quasi-median graph and when the partitions from $\mathfrak{P}$ are given by the hyperplanes of $X$ that cut $X$ into sectors.

\begin{definition}
A \emph{space with partitions} $(X,\mathfrak{P})$ is the data of a set $X$ and a collection of partitions $\mathfrak{P}$ satisfying the following conditions:
\begin{itemize}
	\item for every $\mathcal{P}\in \mathfrak{P}$, $\# \mathcal{P} \geq 2$ and $\emptyset \notin \mathcal{P}$;
	\item for all distinguishable partitions $\mathcal{P},\mathcal{Q} \in \mathfrak{P}$, if there exist two sectors $A\in \mathcal{P}$ and $B \in \mathcal{Q}$ such that $A \subset B$, then $\mathcal{P}$ and $\mathcal{Q}$ are nested;
	\item any two points of $X$ are separated by only finitely many partitions of $\mathfrak{P}$. 
\end{itemize}
\end{definition}

\noindent
For instance, a quasi-median graph endowed with the set of partitions described above defines a space with partitions. 

\begin{definition}
Let $(X,\mathfrak{P})$ be a space with partitions. An \emph{orientation} $\sigma$ is map $\mathfrak{P} \to \{ \text{sectors}\}$ satisfying the following conditions:
\begin{itemize}
	\item $\sigma(P) \in \mathcal{P}$ for every $\mathcal{P} \in \mathfrak{P}$;
	\item $\sigma(\mathcal{P}) \cap \sigma(\mathcal{Q}) \neq \emptyset$ for all $\mathcal{P},\mathcal{Q} \in \mathfrak{P}$.
\end{itemize} 
For every point $x \in X$, 
$$\sigma_x : \mathcal{P} \mapsto \text{sector in $\mathcal{P}$ containing $x$}$$
is the \emph{principal orientation at $x$}.
\end{definition}

\noindent
Roughly speaking, an orientation is a coherent choice of a sector for each partition.

\begin{definition}
Let $(X,\mathfrak{P})$ be a space with partitions. The \emph{quasi-cubulation} $\mathrm{QM}(X,\mathfrak{P})$ is the connected component containing the principal orientations of the graph whose vertices are the orientations and whose edges connect two orientations whenever they differ on a single partition.
\end{definition}

\noindent
The statement below summarises the basic properties satisfied by the quasi-cubulation of a space with partitions. See \cite[Proposition~5.46, Theorem~2.56 and its proof, Lemma~2.60, Corollary~2.51]{QM} for more details.

\begin{thm}\label{thm:Popset}
Let $(X,\mathfrak{P})$ be a space with partitions and let $\iota : X \to \mathrm{QM}$ denote the canonical map from $X$ to the quasi-cubulation $\mathrm{QM}:= \mathrm{QM}(X,\mathfrak{P})$. The following assertions hold:
\begin{itemize}
	\item $\mathrm{QM}$ is a quasi-median graph;
	\item the distance between two orientations $\mu,\nu \in \mathrm{QM}$ coincides with the number of partitions of $\mathfrak{P}$ on which they differ;
	\item the hyperplanes of $\mathrm{QM}$ are $$J_\mathcal{P}:= \left\{ \{\mu,\nu \} \mid \mu(\mathcal{P}) \neq \nu(\mathcal{P}) \text{ but } \mu(\mathcal{Q})= \nu(\mathcal{Q}) \text{ for every } \mathcal{Q} \in \mathfrak{P} \backslash \{ \mathcal{P} \} \right\}, \ \mathcal{P} \in \mathfrak{P};$$
	\item the map $\mathcal{P} \mapsto J_\mathcal{P}$ induces a bijection from $\mathfrak{P}$ to the hyperplanes of $\mathrm{QM}$ that preserves (non-)transversality and (non-)tangency.
\end{itemize}
\end{thm}

\subsection{Conspicial actions}\label{section:Conspicial}

\noindent
The main reason we introduce quasi-median graphs is that, according to \cite{MR4586831}, one can construct an embedding into a graph product of groups thanks to a good action on a quasi-median graph. In this section, we recall that relevant definitions and statements.

\begin{definition}[\cite{MR4586831}]
Let $G$ be a group acting faithfully on a quasi-median graph $X$. The action is \emph{conspicial} if:
\begin{itemize}
	\item for every hyperplane $J$ and every $g \in G$, $J$ and $gJ$ are neither transverse nor tangent;
	\item for all transverse hyperplanes $J_1,J_2$ and for every $g \in G$, $J_1$ and $gJ_2$ are not tangent;
	\item for every hyperplane $J$, $\mathfrak{S}(J)$ acts freely on $\mathscr{S}(J)$. 
\end{itemize}
\end{definition}

\noindent
Given a hyperplane $J$, $\mathscr{S}(J)$ denotes the set of the sectors delimited by $J$ and $\mathfrak{S}(J)$ denotes the image of $\mathrm{stab}(J)$ in the permutation group of $\mathscr{S}(J)$. 

\begin{thm}[\cite{MR4586831}]\label{thm:ConspicialQM}
Let $G$ be a group acting conspicially on a quasi-median graph~$X$. 
\begin{itemize}
	\item Fix representatives $\{J_i \mid i \in I\}$ of hyperplanes of $X$ modulo the acting of $G$.
	\item Let $\Gamma$ denote the graph whose vertex-set is $\{J_i \mid i \in I\}$ and whose edges connect two hyperplanes whenever they have transverse $G$-translates.
	\item For every $i \in I$, let $G_i$ denote the group $\mathfrak{S}(J_i) \oplus K_i$, where $K_i$ is an arbitrary group of cardinality the number of orbits of $\mathfrak{S}(J_i) \curvearrowright \mathscr{S}(J_i)$.
\end{itemize}
Then $G$ embeds as a virtual retract into the graph product $\Gamma \mathcal{G}$ where $\mathcal{G}= \{G_i \mid i \in I\}$.
\end{thm}

\subsection{Mediangle graphs}\label{section:Mediangle}

\noindent
Our definition of \emph{mediangle graphs} is a variation of the definition of quasi-median graphs as weakly modular graphs with no induced copy of $K_{3,2}$ and $K_4^-$; see Section~\ref{section:QM}. Roughly speaking, we replace in this definition $4$-cycles with convex even cycles. In the definition below, $I(\cdot,\cdot)$ denotes the \emph{interval} between two vertices, i.e.\ the union of all the geodesics connecting the two vertices under consideration. 

\begin{definition}\label{def:Mediangle}
A connected graph $X$ is \emph{mediangle} if the following conditions are satisfied:
\begin{description}
	\item[(Triangle Condition)] For all vertices $o,x,y \in X$ satisfying $d(o,x)=d(o,y)$ and $d(x,y)=1$, there exists a common neighbour $z \in X$ of $x,y$ such that $z \in I(o,x) \cap I(o,y)$.
	\item[(Intersection of Triangles)] $X$ does not contain an induced copy of $K_4^-$.
	\item[(Cycle Condition)] For all vertices $o ,x,y,z \in X$ satisfying $d(o,x)=d(o,y)=d(o,z)-1$ and $d(x,z)=d(y,z)=1$, there exists a convex cycle of even length that contains the edges $[z,x],[z,y]$ and such that the vertex opposite to $z$ belongs to $I(o,x) \cap I(o,y)$.
	\item[(Intersection of Even Cycles)] The intersection between any two convex cycles of even lengths contains at most one edge. 
\end{description}
\end{definition}

\noindent
Examples of mediangle graphs include (quasi-)median graphs and Cayley graphs of Coxeter graphs (e.g.\ the one-skeleton of the regular tiling of the plane by hexagons). We refer to \cite{Mediangle} for more examples.

\paragraph{Hyperplanes.} Similarly to (quasi-)median graphs, there is a notion of \emph{hyperplanes} in mediangle graphs that plays a fundamental role.

\begin{definition}
Given a mediangle graph $X$, a \emph{hyperplane} is an equivalence class of edges with respect to the reflexive-transitive closure of the relation that identifies any two edges that belong to a common $3$-cycle or that are opposite in a convex even cycle. Given a hyperplane $J$, we denote by $X \backslash \backslash J$ the graph with the same vertices as $X$ and whose edges are those of $X$ that do not belong to $J$. The connected components of $X \backslash \backslash J$ are the \emph{sectors delimited by $J$}. The \emph{carrier} $N(J)$ of $J$ is the union of all the convex cycles containing edges in $J$, and the connected components of $N(J) \backslash \backslash J$ are the \emph{fibres}. Two vertices of $X$ are \emph{separated by $J$} if they lie in distinct sectors. Two hyperplanes are \emph{transverse} if they contain two distinct pairs of opposite edges in some convex even cycle. They are \emph{tangent} if they are distinct, not transverse, but not separated by a third hyperplane.
\end{definition}

\begin{figure}
\begin{center}
\includegraphics[width=0.7\linewidth]{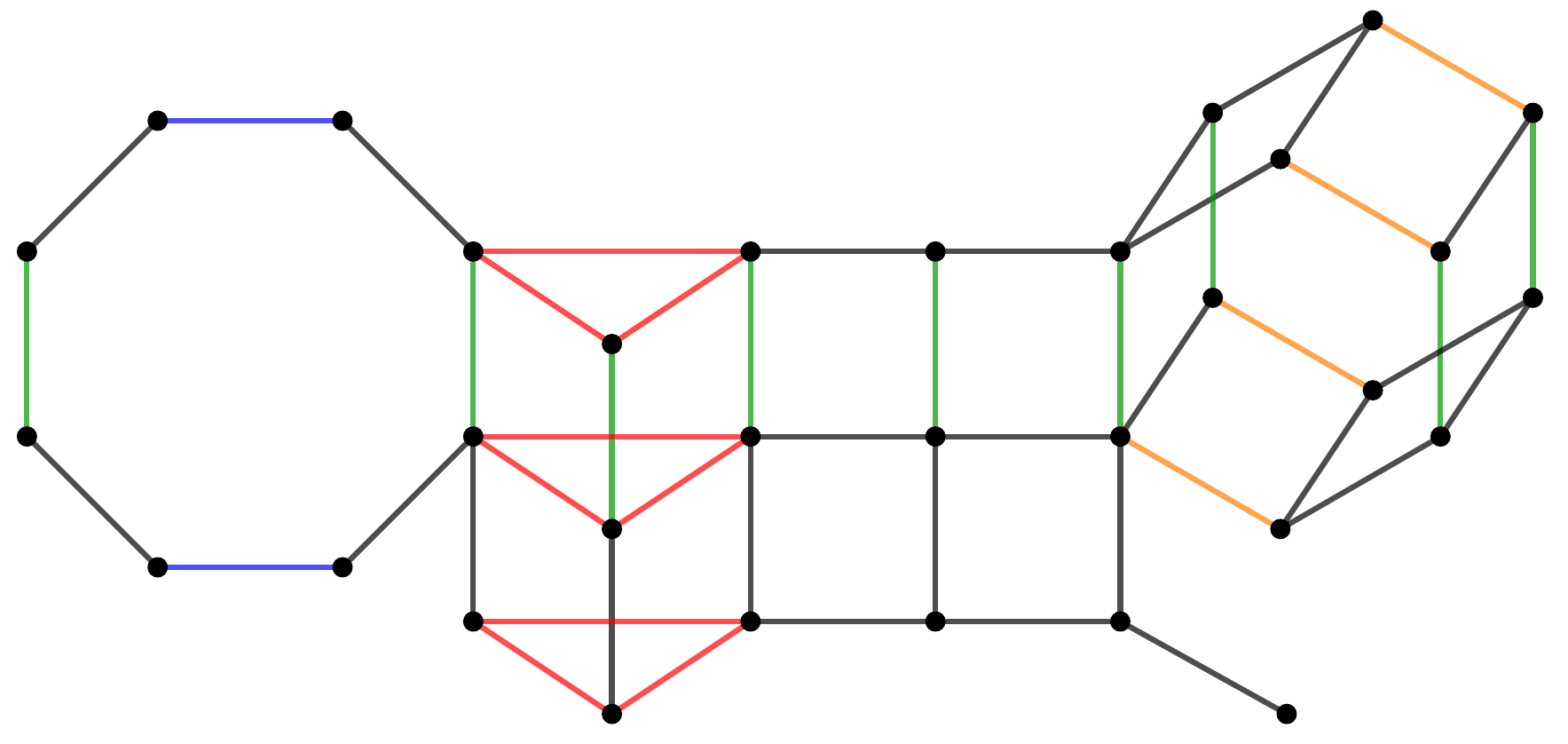}
\caption{A mediangle graph of some of its hyperplanes. The green hyperplane is transverse to the orange, red, and blue hyperplanes. The red and blue hyperplanes are tangent. The orange and red hyperplanes are neither transverse nor tangent.}
\label{MediangleHyp}
\end{center}
\end{figure}

\noindent
See Figure~\ref{MediangleHyp} for an example. The next statement, proved in \cite[Theorem~3.9]{Mediangle}, summarises the basic properties of hyperplanes in mediangle graphs. 

\begin{thm}\label{thm:BigHyp}
Let $X$ be a mediangle graph. The following assertions hold.
\begin{itemize}
	\item[(i)] Hyperplanes separate $X$. More precisely, given a hyperplane $J$ and a clique $C \subset J$, any two vertices in $C$ are separated by $J$ and every sector delimited by $J$ contains a vertex of $C$. 
	\item[(ii)] Sectors delimited by hyperplanes are always convex.
	\item[(iii)] A path is a geodesic if and only if it crosses each hyperplane at most once. As a consequence, the distance between any two vertices coincides with the number of hyperplanes separating them.
\end{itemize}
\end{thm}

\noindent
We emphasize that, contrary to (quasi-)median graphs, in mediangle graphs sectors may not be gated and carriers may not be convex.

\begin{definition}
Let $G$ be a group acting on a mediangle graph $X$. The \emph{rotative-stabiliser} $\mathrm{stab}_\circlearrowright(J)$ of a hyperplane $J$ is the subgroup of $\mathrm{stab}(J)$ that stabilises every clique in $J$. 
\end{definition}

\noindent
For instance, consider the action of a group $A \times \mathbb{Z}$ on $\mathrm{Cayl}(A \times \mathbb{Z}, A \cup \{1\})$. The Cayley graph is a product between a clique $\mathrm{Cayl}(A,A)$ and a bi-infinite line $\mathrm{Cayl}(\mathbb{Z},\{1\})$. The whole group $A \times \mathbb{Z}$ stabilises the hyperplane $J$ that contains the $\mathrm{Cayl}(A,A)$-factors, but $\mathrm{stab}_\circlearrowright(J)$ reduces to $A \times \{1\}$. 

\medskip \noindent
We record the following easy observation for future use:

\begin{lemma}\label{lem:RotateCloser}
Let $X$ be a mediangle graph and $A \subset X$ a subgraph. Let $x \in X$ be a vertex and $J$ a hyperplane separating $x$ from $A$. If $g \in \mathrm{stab}_{\circlearrowright}(J)$ sends the sector delimited by $J$ containing $x$ to the sector containing $A$, then $d(gx,A)<d(x,A)$. 
\end{lemma}

\begin{proof}
Fix a vertex $a \in A$ satisfying $d(x,A)=d(x,a)$ and a geodesic $\gamma$ connecting $a$ to $x$. Necessary, $\gamma$ crosses $J$, so it can be decomposed as a concatenation $\alpha \cup e \cup \beta$ where $e:=[e_1,e_2]$ is an edge that belongs to $J$. Notice that the initial (resp.\ terminal) endpoint of $e$ belongs to the sector delimited by $J$ that contains $A$ (resp.\ $x$), hence $ge_2=e_1$. In other words, the initial point of $g\beta$ must coincide with the terminal point of $\alpha$. As a consequence, the concatenation $\alpha \cup g \beta$ is well-defined, providing a path of length $\mathrm{length}(\gamma)-1$ connecting $a$ to $gx$. Therefore, we have
$$d(A,gx) \leq d(a,gx) < \mathrm{length}(\gamma) = d(a,x)=d(A,x),$$
concluding the proof of our lemma.
\end{proof}

\paragraph{Projections.} Recall from Section~\ref{section:QM} that, given a gated subgraph $Y$, we refer to the map that associates to every vertex its gate in $Y$ as the projection to $Y$. The next statement, proved in \cite[Lemma~3.17 and Corollary~3.18]{Mediangle}, records the interaction between projections and hyperplanes.

\begin{thm}\label{thm:BigProjection}
Let $X$ be a mediangle graph and $Y \subset X$ a gated subgraph. The following assertions hold.
\begin{itemize}
	\item For every $x \in X$, any hyperplane separating $x$ from its projection on $Y$ separates $x$ from $Y$.
	\item For all $x,y \in X$, the hyperplanes separating their projections on $Y$ are exactly the hyperplanes separating $x$ and $y$ that cross $Y$.
\end{itemize}
\end{thm}

\noindent
As a direct application, let us notice that:

\begin{lemma}\label{lem:GatedDisjoint}
In a mediangle graph, two disjoint gated subgraphs are separated by some hyperplane.
\end{lemma}

\begin{proof}
Let $X$ be a mediangle graph and $Y,Z \subset X$ two disjoint gated subgraphs. Fix two vertices $y \in Y$ and $z \in Z$ minimising the distance between $Y$ and $Z$. Necessarily, $z$ is the projection of $y$ to $Z$ and vice-versa. Applying Theorem~\ref{thm:BigProjection} twice, we deduce that the hyperplanes separating $y$ and $z$ coincide with the hyperplanes separating $Y$ and $Z$. But, since $Y$ and $Z$ are disjoint, $d(y,z) \geq 1$; so there must be at least one hyperplane separating $y$ and $z$. 
\end{proof}

\paragraph{Angles between hyperplanes.} A key difference between (quasi-)median graphs and mediangle graphs is that hyperplanes, in mediangle graphs, can cross in different ways. More precisely, it is possible to define an angle between two transverse hyperplanes.

\begin{definition}
In a mediangle graph, if two transverse hyperplanes $J_1,J_2$ both cross a convex even cycle $C$, the \emph{angle between $J_1,J_2$ at $C$} is
$$\measuredangle_C(J_1,J_2): = 2 \pi \cdot \frac{1+d( J_1 \cap C, J_2 \cap C)}{\mathrm{length}(C)}.$$
\end{definition}

\noindent
This number coincides with the geometric angle between the straight lines connecting the midpoints of the two pairs of opposite edges in $J_1,J_2$ when $C$ is thought of as a regular Euclidean polygon. The key observation is that this angle does not depend on the cycle $C$ under consideration \cite[Proposition~3.20]{Mediangle}, allowing us to define the \emph{angle} $\measuredangle(J_1,J_2)$ as the angle at an arbitrary convex even cycle they both cross.

\medskip \noindent
For instance, on Figure~\ref{MediangleHyp} above, the green hyperplane is transverse to the red and blue hyperplanes with angle $\pi/2$ in both cases, and it is transverse to the orange hyperplane with angle $\pi/3$.

\begin{prop}\label{prop:SpanningACycle}
Let $X$ be a mediangle graph. Let $o \in X$ be a vertex and $a,b \in X$ two neighbours. If the hyperplanes $A$ and $B$ containing respectively the edges $[o,a]$ and $[o,b]$ are transverse, then $o,a,b$ span a convex $2n$-cycle where $n:=\pi/ \measuredangle(A,B)$. 
\end{prop}

\noindent
In order to prove the proposition, we start by proving the following assertion:

\begin{lemma}\label{lem:GatedHullSquare}
Let $X$ be a mediangle graph and $C$ a convex $4$-cycle. The gated hull of $C$ is a product of cliques. Moreover, $C$ and its gated hull are crossed by exactly the same hyperplanes.
\end{lemma}

\begin{proof}
Fix a vertex $o \in C$ and let $a,b \in C$ denote its two neighbours. Let $P$ (resp.\ $Q$) denote the clique containing the edge $[o,a]$ (resp.\ $[o,b]$) and let $A$ (resp.\ $B$) denote the hyperplane containing $P$ (resp.\ $Q$). We begin by proving the following assertion, in which $\mathcal{H}(\cdot,\cdot)$ denotes the set of the hyperplanes separating the two vertices under consideration.

\begin{claim}\label{claim:ConstructingPrism}
For all sectors $R,S$ respectively delimited by $A,B$, there exists a unique vertex $p(R,S) \in R \cap S$ such that $\mathcal{H}(o,p) \subset \{A,B\}$.
\end{claim}

\noindent
For the uniqueness, assume that $p,q \in R \cap S$ are two vertices satisfying $\mathcal{H}(o,p) , \mathcal{H}(o,q) \subset \{A,B\}$. Since
$$\mathcal{H}(p,q) \subset \mathcal{H}(o,p) \triangle \mathcal{H}(o,q) \subset \{A,B\},$$
we know that $A$ and $B$ are the only hyperplanes that may separate $p$ and $q$. But these two vertices belong the same sectors delimited by $A$ and $B$, so neither $A$ nor $B$ can separate $p$ and $q$. Thus, no hyperplane separates $p$ and $q$, which amounts to saying that $p=q$. 

\medskip \noindent
For the existence, let us show that, for every vertex $a'  \in P \backslash \{ o \}$, the edges $[o,a']$ and $[o,b]$ span a $4$-cycle. Let $P'$ denote the clique containing the edge of $C$ opposite to $[o,a]$ and let $b'$ denote the projection of $a'$ on $C'$. Clearly, $b'$ is adjacent to $b$. The path $[a',o] \cup [o,b] \cup [b,b']$ crosses successively the hyperplanes $A$, $B$, and $A$ again. It follows that $a'$ and $b'$ are separated by $B$ and possibly by $A$. But $a'$ and $b'$ have the same projection to the clique $C'$ from $A$ (namely, $b'$ itself), so they cannot be separated by $A$. Therefore, $a'$ and $b'$ are separated by exactly one hyperplane, namely $B$. We conclude that $(o,a',b',b)$ is the $4$-cycle we are looking for. 

\medskip \noindent
By repeating the same argument with $a'$ in place of $b$ and an arbitrary vertex of $Q \backslash \{o\}$ in place of $a'$, we conclude that, for all vertices $a' \in P \backslash \{o\}$ and $b' \in Q \backslash \{o\}$, the edges $[o,a']$ and $[o,b']$ span a $4$-cycle. This proves Claim~\ref{claim:ConstructingPrism} for the sectors delimited by $A$ and $B$ that do not contain $o$. Indeed, if $R_o$ (resp.\ $S_o$) denotes the sector delimited by $A$ (resp.\ $B$) that contains $o$, then, for all sectors $R \neq R_o$ and $S \neq S_o$ delimited by $A$ and $B$, the fourth vertex of the $4$-cycle spanned by the edges $[o,a']$ and $[o,b']$, where $a'$ (resp.\ $b'$) is the unique vertex of $P \cap R$ (resp.\ $Q \cap S$), is separated from $o$ by the hyperplanes $A$ and $B$ exactly. Finally, notice that $a' \in R \cap S_o$ is separated from $o$ by $A$ exactly; that $b' \in R_o \cap S$ is separated from $o$ by $B$ exactly; and that $o \in R_o \cap S_o$ is separated by $o$ by no hyperplane. This concludes the proof of Claim~\ref{claim:ConstructingPrism}.

\medskip \noindent
Now, set
$$K:= \{ p(R,S) \mid \text{$R$ and $S$ sectors delimited by $A$ and $B$}\}$$
and let us prove that $K$ decomposes as the product of $P$ and $Q$. More precisely, we claim that:

\begin{claim}\label{claim:ThisIsPrism}
The map $P \times Q \to K$ defined by $(x,y) \mapsto p(R_x,S_y)$ induces a graph isomorphism, where $R_x$ (resp.\ $S_y$) denotes the sector delimited by $A$ (resp.\ $B$) containing $x$ (resp.\ $y$). 
\end{claim}

\noindent
Let $x,x' \in P$ and $y,y' \in Q$ be four vertices. Since
$$\mathcal{H}(p(R_x,S_y),p(R_{x'},S_{y'})) \subset \mathcal{H}(o,p(R_x,S_y)) \triangle \mathcal{H}(o, p(R_{x'},S_{y'})) \subset \{A,B\},$$
only the hyperplanes $A$ and $B$ may separate $p(R_x,S_y)$ and $p(R_{x'},S_{y'})$. It follows that
$$\mathcal{H}(p(R_x,S_y), p(R_{x'},S_{y'})) = \left\{ \begin{array}{cl} \{A,B\} & \text{if $x \neq x'$ and $y \neq y'$} \\ \{A\} & \text{if $x\neq x'$ and $y=y'$} \\ \{B\} & \text{if $x=x'$ and $y \neq y'$} \\ \emptyset & \text{if $x=x'$ and $y=y'$} \end{array} \right.,$$
from which we deduce that
$$d(p(R_x,S_y),p(R_{x'},S_{y'})) = d(x,x')+d(y,y') = d((x,x'),(y,y')).$$
This concludes the proof of Claim~\ref{claim:ThisIsPrism}.

\medskip \noindent
So far, we have constructed a product of cliques containing $C$, namely $K$. Clearly, $A$ and $B$ are the only hyperplanes crossing $K$, i.e.\ $C$ and $K$ are crossed by exactly the same hyperplanes. It remains to show that $K$ is the gated hull of $C$. The fact that $K$ is contained in the gated hull of $C$ is rather clear. Indeed, because a clique containing at least one edge in a gated subgraph must be contained entirely in this subgraph, we know that $P$ and $Q$ are contained in the gated hull of $C$. Next, we know that a $4$-cycle having two consecutive edges in a gated subgraph must be contained entirely in this subgraph. But every vertex of $K$ not in $P \cup Q$ is the fourth vertex of a $4$-cycle spanned by two edges in $P \cup Q$. (This is clear by identifying $K$ with $P \times Q$, as justified by Claim~\ref{claim:ThisIsPrism}.) Therefore, in order to conclude the proof of our lemma, it suffices to show that:

\begin{claim}\label{claim:KisGated}
$K$ is gated.
\end{claim}

\noindent
Let $x \in X$ be an arbitrary vertex. Let $R$ (resp.\ $S$) denote the sector delimited by $A$ (resp.\ $B$) containing $x$. Let us verify that $p:=p(R,S)$ is a gate for $x$. Fix an arbitrary vertex $q \in K$. A hyperplane separating $p$ and $q$ must cross $K$, hence $\mathcal{H}(p,q) \subset \{A,B\}$. On the other hand, since $x$ and $p$ belong to the same sector delimited by $A$ and $B$, necessarily neither $A$ nor $B$ can separate $x$ and $p$. In other words, $\mathcal{H}(p,q) \cap \mathcal{H}(x,p)= \emptyset$. This implies that $d(x,q)=d(x,p)+d(p,q)$, as desired. This concludes the proof of Claim~\ref{claim:KisGated}.
\end{proof}

\begin{proof}[Proof of Proposition~\ref{prop:SpanningACycle}.]
Because $A$ and $B$ are transverse, there exists a convex even cycle $C$ that contains two pairs of opposite edges in $A$ and $B$. We distinguish two cases depending on the perimeter of $C$.

\medskip \noindent
First, we assume that $C$ has perimeter $>4$. According to \cite[Theorem~3.5]{Mediangle}, $C$ is gated, so we can consider the projections $a'$ and $b'$ of $a$ and $b$ to $C$. According to Theorem~\ref{thm:BigProjection}, the hyperplanes separating $a'$ and $b'$ are exactly the hyperplanes separating $a$ and $b$ that cross $C$. It follows that $A$ and $B$ are the only hyperplanes separating $a'$ and $b'$, hence $d(a',b')=2$. Let $o'$ denote the unique common neighbour of $a'$ and $b'$. By convexity of $C$, necessarily $o' \in C$. Let $o'' \in C$ denote the vertex opposite to $o'$. 

\medskip \noindent
Notice that the hyperplanes separating $a$ and $a'$ coincide with the hyperplanes separating $b$ and $b'$. Indeed, a hyperplane separating $a$ from $a'$ has to separate $a$ from $C$, so it has to separate $b$ from $b'$ or $a$ from $b$. But the latter case is not possible since the hyperplanes separating $a$ and $b$, namely $A$ and $B$, cross $C$. Thus, every hyperplane separating $a$ from $a'$ separates $b$ from $b'$, and the converse holds by symmetry. We conclude that $d(a,a')=d(b,b')$.

\medskip \noindent
On the other hand, because $a'$ is the projection of $a$ to $C$, we have
$$d(a,o'')=d(a,a')+d(a',o'')= d(a,a') + |C|/2-1.$$
Similarly, $d(b,o'')= d(b,b')+ |C|/2-1$. So our previous observation implies that $d(a,o'')=d(b,o'')$.

\medskip \noindent
Also, notice that $d(o,o'')=d(a,o'')+1$. Indeed, the path obtained by concatenating the edge $[o,a]$ with a geodesic between $a$ and $o''$ does not cross a hyperplane twice. Otherwise, such a hyperplane would have to be $A$, but it does not separate $a$ from $o''$. Thus, we are in position to apply the Cycle Condition to $o'',a,b,o$ and to deduce that the edge $[o,a], [o,b]$ span a convex even cycle, say $Q$. Because
$$\frac{2\pi}{|Q|} = \measuredangle_Q(A,B)= \measuredangle (A,B)= \frac{\pi}{n},$$
we conclude that $Q$ is a $2n$-cycle, as desired.

\medskip \noindent
Next, assume that $C$ is a $4$-cycle. Our strategy will be essentially the same, but we have to adapt the execution since $4$-cycles may not be gated. Let $C^+$ denote the gated hull of $C$. According to Lemma~\ref{lem:GatedHullSquare}, $C^+$ is a product of two cliques contained in $A$ and $B$. Let $o',a',b'$ denote the projections of $o,a,b$ on $C^+$. According to Theorem~\ref{thm:BigProjection}, the hyperplanes separating $o'$ and $a'$ coincide with the hyperplanes separating $o$ and $a$ that cross $C^+$. Therefore, $A$ is the only hyperplane separating $o'$ and $a'$, proving that $o'$ and $a'$ are adjacent. Similarly, $B$ is the only hyperplane separating $o'$ from $b'$ and $o'$ is adjacent to $b'$. Since $C^+$ is a product of two cliques, we know that the two edges $[o',a']$ and $[o',b']$ span a $4$-cycle in $C^+$. Let $o''$ denote the the fourt vertex of this cycle. It is separated from $a'$ by $B$, from $b'$ by $A$, and from $o'$ by $A,B$. 

\medskip \noindent
Notice that $\mathcal{H}(a,a')= \mathcal{H}(b,b')= \mathcal{H}(o,o')$, where $\mathcal{H}(\cdot,\cdot)$ denote the set of the hyperplanes separating two given vertices. Indeed, if a hyperplane separating $a$ from $a'$, then it has to separate $a$ from $C^+$. In particular, it has to be distinct from $A$ and $B$, which implies that it cannot separate $a$ from $o$ or $b$ nor $a'$ from $o'$ or $b'$. Necessarily, it has to separate $o$ from $o'$ and $b$ from $b'$. The other inclusions follow by symmetry. Let $\mathcal{H}$ denote this common set of hyperplanes. We have
$$\begin{array}{l} d(o'',a)= \# \mathcal{H}(o'',a) = \# (\mathcal{H} \cup \{B\}) = 1+ \# \mathcal{H} \\
d(o'',b)= \# \mathcal{H}(o'',b) = \# (\mathcal{H} \cup \{A\}) = 1+ \# \mathcal{H} \\
d(o'',o)= \# \mathcal{H}(o'',o)= \# ( \mathcal{H} \cup \{A,B\})= 2+ \# \mathcal{H} \end{array}$$
Therefore, we can apply the Cycle Condition in order to deduce that the edges $[o,a]$ and $[o,b]$ span a convex even cycle, say $Q$. Because
$$\frac{2\pi}{|Q|} = \measuredangle_Q(A,B)= \measuredangle_C (A,B)= \frac{\pi}{2},$$
we conclude that $Q$ is a $4$-cycle, as desired.
\end{proof}

\paragraph{Right hyperplanes.} Given two transverse hyperplanes in a mediangle graph, there is a true difference depending on whether the angle between the two hyperplanes is $\pi/2$. Usually, having right angles is simpler. 

\begin{definition}
In a mediangle graph, a hyperplane $J$ is \emph{right} if $\measuredangle(J,H) = \pi/2$ for every hyperplane $H$ transverse to $J$.
\end{definition}

\noindent
We conclude this section with a few observations about right hyperplanes.

\begin{lemma}\label{lem:NotRight}
In a mediangle graph, a hyperplane that is not right delimits exactly two sectors. 
\end{lemma}

\begin{proof}
Let $X$ be a mediangle graph and $J$ a hyperplane. If $J$ is not right, then there exists some hyperplane $H$ transverse to $J$ such that $\measuredangle(J,H) \neq \pi/2$. This implies that $J$ crosses a convex even cycle $C$ of length $>4$. If $J$ delimits more than two sectors, then an edge $e$ from $J \cap C$ must belong to a clique (contained in $J$) of size $\geq 3$. In particular, there must be a $3$-cycle in $J$ containing $e$. But this contradicts the fact that $C$, as any convex even cycle of length $>4$, is gated \cite[Theorem~3.5]{Mediangle}.
\end{proof}

\begin{lemma}\label{lem:RightAngledCommutation}
Let $X$ be a mediangle graph and $J,H$ two transverse hyperplanes. If $\measuredangle (J,H)=\pi/2$, then every isometry in $\mathrm{stab}_{\circlearrowright}(J)$ stabilises $H$.
\end{lemma}

\begin{proof}
Because $\measuredangle (J,H)= \pi/2$, we can find a $4$-cycle $(a,b,c,d)$ such that the edges $[a,b], [c,d]$ belong to $J$ and the edges $[b,c],[a,d]$ to $H$. Let $A$ denote the clique containing the edge $[a,b]$ and $C$ the clique containing the ege $[c,d]$. Because an isometry $g \in \mathrm{stab}_\circlearrowright(J)$ stabilises both $A$ and $C$, the edges $[a,d]$ and $g[a,d]$ both connect $A$ and $C$. But, according to \cite[Lemma~3.6]{Mediangle}, the union $A \cup C$ decomposes as a product between $A$ and an edge, which implies that $[a,d]$ and $g[a,d]$ belong to the same hyperplane, namely $H$. Hence $gH=H$, as desired. 
\end{proof}

\begin{lemma}\label{lem:RightHypGated}
In a mediangle graph, the carrier of a right hyperplane is gated.
\end{lemma}

\begin{proof}
Let $J$ be a right hyperplane. We already know from Theorem~\ref{thm:BigHyp} that the carrier of $J$ is connected. According to \cite[Proposition~6.5]{Mediangle}, it then suffices to show that $N(J)$ is \emph{locally gated}, i.e.\ it satisfies the following two conditions:
\begin{itemize}
	\item every $3$-cycle having an edge in $N(J)$ lies entirely in $N(J)$;
	\item every convex even cycle having two consecutive edges in $N(J)$ lies entirely in $N(J)$.
\end{itemize}
Let $(a,b,c)$ be a $3$-cycle with $a,b \in N(J)$. If the edge $[a,b]$ belongs to $J$, then clearly $c \in N(J)$. Otherwise, there is a $4$-cycle $(a,b,x,y)$ in $N(J)$ with the edges $[a,y]$ and $[b,x]$ in $J$. Clearly, $J$ is transverse to the hyperplane $H$ containing the $3$-cycle $(a,b,c)$ and $\measuredangle (J,H)=\pi/2$. It follows from Proposition~\ref{prop:SpanningACycle} that the edges $[a,y]$ and $[a,c]$ span a $4$-cycle. Let $z$ denote the fourth vertex of this cycle. The edges $[c,z]$ necessarily belongs to $J$, which implies that $c \in N(J)$, as desired.

\medskip \noindent
Next, let $C=(a_0,a_1, \ldots, a_{2n-1})$ be a convex even cycle with $a_0,a_1,a_2 \in N(J)$. If one of the edges of $C$ belongs to $N(J)$, then necessarily $C \subset N(J)$ and there is nothing else to prove. So we assume that this is not the case. Also, if $n=2$, then the inclusion $C \subset N(J)$ follows from the convexity given by Theorem~\ref{thm:BigHyp}, so we assume that $n>2$. First, it follows that there exists a $4$-cycle $(a_0,a_1,a_1',a_0')$ such that the edges $[a_0,a_0']$ and $[a_1,a_1']$ belong to $J$. Next, a similar $4$-cycle must exist for $[a_1,a_2]$. Since $J$ is necessarily transverse to the hyperplane containing $[a_1,a_2]$, we deduce from Proposition~\ref{prop:SpanningACycle} that this cycle can be chosen as $(a_1,a_2,a_2',a_1')$ with the same $a_1'$ as before. Since the hyperplanes containing $[a_0',a_1']$ and $[a_1',a_2']$ are transverse (they meet in $C$), Proposition~\ref{prop:SpanningACycle} implies that these two edges span a convex $2n$-cycle $C'=(a_0',a_1', \ldots, a_{2n-1}')$. Since the projection of $C$ on $C'$ contains at least two edges (namely, $[a_0',a_1']$ and $[a_1',a_2']$), it follows from \cite[Proposition~3.21]{Mediangle} that $C$ projects isometrically onto $C'$. As a consequence of Theorem~\ref{thm:BigProjection}, the hyperplanes crossing $C$ have to cross $C'$ as well. But $J$ clearly separates $C$ and $C'$, so we conclude that $J$ is transverse to all the hyperplanes crossing $C$ (with right angles since $J$ is a right hyperplane). Therefore, by applying Proposition~\ref{prop:SpanningACycle} iteratively, we find neighbours $a_3'',\ldots, a_{2n-1}''$ respectively of $a_3, \ldots, a_{2n-1}$ such that $a_0',a_1',a_2',a_3'', \ldots, a_{2n-1}''$ is a path. From this configuration, we see that all the edges $[a_i,a_i'']$ for $3 \leq i \leq 2n-1$ belong to $J$, proving that $C \subset N(J)$, as desired.
\end{proof}

\subsection{Mediangle geometry of periagroups}\label{section:Periagroups}

\noindent
Periagroups, as defined in the introduction by Definition~\ref{def:Periagroups}, generalise graph products of groups and Coxeter groups. As shown in \cite{Mediangle}, a good geometric framework in the study of periagroups is given by mediangle graphs. 

\begin{thm}[\cite{Mediangle}]
Let $\Pi:= \Pi(\Gamma, \mathcal{G},\lambda)$ be a periagroup. The Cayley graph
$$M(\Gamma, \mathcal{G},\lambda):= \mathrm{Cayl}\left( \Pi, \bigcup\limits_{G\in \mathcal{G}} G \backslash \{1\} \right)$$
is mediangle. 
\end{thm}

\noindent
We emphasize that, since vertex-groups in our periagroups are allowed to be infinite, mediangle graphs given by periagroups may not be locally finite. 

\medskip \noindent
\textbf{Convention:} When fixing an arbitrary periagroup, we will always assume that its vertex-groups are non-trivial. 

\medskip \noindent
Of course, there is no loss of generality since a periagroup with some of its vertex-groups trivial can be easily described as a new periagroup all of whose vertex-groups are non-trivial. 

\medskip \noindent
In the sequel, we will use the following notation. Given a periagroup $\Pi(\Gamma, \mathcal{G},\lambda)$ and a subgraph $\Xi \subset \Gamma$, we denote by $\langle \Xi \rangle$ the subgroup of $\Pi(\Gamma, \mathcal{G},\lambda)$ generated by the vertex-groups indexed by the vertices of $\Xi$.

\paragraph{Right hyperplanes.} In the rest of this section, we record a few properties of right hyperplanes in mediangle graphs of periagroups.

\begin{lemma}\label{lem:NoObsGP}
Let $\Pi:= \Pi(\Gamma, \mathcal{G},\lambda)$ be a periagroup and let $M:= M(\Gamma, \mathcal{G},\lambda)$ denote its mediangle graph. Fix a vertex $u \in V(\Gamma)$ and a non-trivial element $s \in G_u$. The hyperplane $J_u$ containing the edge $[1,s]$ is right if and only if every edge of $\Gamma$ containing $u$ is labelled by $2$.
\end{lemma}

\begin{proof}
Let $C$ be a convex even cycle containing two opposite edges in $J_u$. For simplicity, assume that $1 \in C$. Let $r \in C$ denote the neighbour of $1$ in $C$ distinct from $s$. Let $v \in V(\Gamma)$ be such that $r \in G_v$. Because $[1,r]$ and $[1,s]$ span the cycle $C$, necessarily $u$ and $v$ are adjacent in $\Gamma$. Moreover, the edge connecting $u$ and $v$ is labelled by $|C|/2$. Conversely, if $w\in V(\Gamma)$ is a neighbour of $u$, then fixing a non-trivial element $t \in G_w$, the edges $[1,s]$ and $[1,t]$ span a convex even cycle of length $2\lambda(\{u,w\})$. 

\medskip \noindent
We deduce from this observation that the convex even cycles crossed by $J_u$ all have length four if and only if every edge of $\Gamma$ containing $u$ is labelled by $2$. This characterisation amounts to saying that $J_u$ is right.
\end{proof}

\begin{lemma}\label{lem:CarrierRightHypPeriagroup}
Let $\Pi:= \Pi(\Gamma, \mathcal{G},\lambda)$ be a periagroup and let $M:= M(\Gamma, \mathcal{G},\lambda)$ denote its mediangle graph. Fix a vertex $u \in \Gamma$ such that every edge containing $u$ has label $2$ and let $J_u$ denote the hyperplane of $M$ containing the clique $\langle u \rangle$. Then $N(J_u)= \langle \mathrm{star}(u) \rangle$.
\end{lemma}

\begin{proof}
Let $x \in N(J_u)$ be a vertex. There must exist a sequence of edges $e_0, \ldots, e_n$ such that $e_0 \subset \langle u \rangle$; $e_n$ contains $x$; and, for every $0 \leq i \leq n-1$, $e_i$ and $e_{i+1}$ are opposite sides of some $4$-cycle or they belong to the same $3$-cycle. But 
\begin{itemize}
	\item a $4$-cycle in $M$ is necessarily a $\Pi$-translate of a $4$-cycle of the form $(1,r, rs=sr, s)$ where $r$ and $s$ are two non-trivial elements of vertex-groups $G_v$ and $G_w$ with $v,w$ connected in $\Gamma$ by an edge whose label is $2$;
	\item and a $3$-cycle in $M$ is necessarily a $\Pi$-translate of a $3$-cycle of the form $(1,r,s)$ where $r$ and $s$ are distinct and non-trivial elements of a vertex group $G_w$.
\end{itemize}
Therefore, for every $0 \leq i \leq n-1$, the edges $e_i$ and $e_{i+1}$ are labelled by generators of $\Pi$ coming from $G_u$ and their endpoints are connected by paths whose edges are labelled by generators coming from $\langle \mathrm{star}(u) \rangle$. We conclude that $x \in \langle \mathrm{star}(u) \rangle$, as desired.

\medskip \noindent
Conversely, let $x \in \langle \mathrm{star}(u) \rangle$ be a vertex. Since $\langle \mathrm{star}(u) \rangle$ decomposes as $\langle \mathrm{link}(u) \rangle \times G_u$, we can write $x$ as $yz$ with $y \in \langle \mathrm{link}(u) \rangle$ and $z \in G_u$. Decompose $y$ as a product $y_1 \cdots y_n$ of generators coming from $\langle \mathrm{link}(u) \rangle$. Notice that
$$[1,z], \ [y_1,y_1z], \ [y_1y_2, y_1y_2z], \ \ldots, \ [y_1 \cdots y_n, y_1 \cdots y_nz] = [y,yz]$$
is a sequence of edges such that any two consecutive edges are opposite sides in a $4$-cycle. Since $[1,z] \subset \langle u \rangle$ is contained in $J_u$, it follows that $[y,yz] \subset N(J_u)$, and a fortiori $x=yz \in N(J_u)$. 
\end{proof}

\noindent
In Lemma~\ref{lem:CarrierRightHypPeriagroup}, notice that the cliques contained in $J_u$ are given by the $\langle u \rangle$-factors in the decomposition $\langle \mathrm{star}(u) \rangle = \langle \mathrm{link}(u) \rangle \times \langle u \rangle$. As a consequence, all the edges of a right hyperplane are labelled by generators coming from the same vertex-group. This allows us to label the right hyperplanes of the mediangle graph $M(\Gamma, \mathcal{G},\lambda)$ by the vertices in of $\Gamma$. Our next lemma records a useful property satisfied by these labels.

\begin{lemma}\label{lem:LabelHypGP}
Let $\Pi:= \Pi(\Gamma, \mathcal{G},\lambda)$ be a periagroup and let $M:= M(\Gamma, \mathcal{G},\lambda)$ denote its mediangle graph. If two right hyperplanes of $M$ are transverse (resp.\ tangent), then their labels are adjacent (resp.\ distinct and non-adjacent).
\end{lemma}

\begin{proof}
Let $J$ and $H$ be two right hyperplanes, respectively labelled by the vertices $u$ and $v$ of $\Gamma$. If $J$ and $H$ are transverse, then there must exist a $4$-cycle crossed by both $J$ and $H$. But a $4$-cycle in $M$ is necessarily a $\Pi$-translate of a $4$-cycle of the form $(1,r, rs=sr, s)$ where $r$ and $s$ are two non-trivial elements of vertex-groups $G_x$ and $G_y$ with $x,y$ connected in $\Gamma$ by an edge whose label is $2$. This implies that $u$ and $v$ must be adjacent in $\Gamma$. Next, assume that $J$ and $H$ are tangent. Since $N(J)$ and $N(H)$ are gated according to Lemma~\ref{lem:RightHypGated}, we deduce from Lemma~\ref{lem:GatedDisjoint} that, because $N(J)$ and $N(H)$ cannot be separated by a hyperplane, $N(J) \cap N(H) \neq \emptyset$. Fix a vertex $o \in N(J) \cap N(H)$ and two neighbours $a,b \in M$ such that $[o,a]$ belongs to $J$ and $[o,b]$ to $H$. Because $J$ and $H$ are distinct (resp.\ not transverse), the edges $[o,a]$ and $[o,b]$ cannot be labelled by generators coming from the same vertex-group (resp.\ from adjacent vertex-groups), otherwise $[o,a]$ and $[o,b]$ would belong to a $3$-cycle (resp.\ to a convex even cycle). This shows that $u$ and $v$ must be distinct and not adjacent in $\Gamma$. 
\end{proof}

\section{Separability in periagroups}\label{section:SeparabilityPeriagroups}

\noindent
In this section, we prove that double cosets of standard parabolic subgroups in residually finite periagroups are separable (i.e.\ closed in the profinite topology). Namely:

\begin{thm}\label{thm:DoubleCoset}
Let $\Pi:= \Pi(\Gamma, \mathcal{G},\lambda)$ be a periagroup with $\Gamma$ finite. Assume that $\Pi$ is residually finite. For all $\Phi, \Psi \subset V(\Gamma)$, the double coset $\langle \Phi \rangle \langle \Psi \rangle$ is separable in $\Pi$. 
\end{thm}

\noindent
Notice that, as a consequence of Theorem~\ref{thm:BigIntro}, a periagroup is residually finite if and only if its vertex-groups are residually finite. However, this is not something we need to know now.

\medskip \noindent
As a first step towards the proof of Theorem~\ref{thm:DoubleCoset}, we start by considering the case of parabolic subgroups.

\begin{prop}\label{prop:ParabolicSeparable}
Let $\Pi : = \Pi(\Gamma, \mathcal{G}, \lambda)$ be a periagroup. Assume that $\Pi$ is residually finite. For every $\Xi \subset V(\Gamma)$, the standard parabolic subgroup $\langle \Xi \rangle$ is separable.
\end{prop}

\noindent
The strategy to prove Proposition~\ref{prop:ParabolicSeparable} is essentially to construct virtual retracts by playing ping-pong with rotative-stabilisers of hyperplanes. As a preliminary observation, let us recall from \cite[Corollary~6.4]{Mediangle} that rotative-stabilisers in our mediangle graphs have a sufficiently rich dynamics:

\begin{lemma}\label{lem:RotativeStabPeria}
Let $\Pi:= \Pi(\Gamma, \mathcal{G},\lambda)$ be a periagroup. Let $M:= M(\Gamma, \mathcal{G},\lambda)$ denote the mediangle graph on which it acts. For every hyperplane $J$ of $M$, the rotative-stabiliser $\mathrm{stab}_\circlearrowright (J)$ permutes freely and transitively the sectors delimited by $J$. \qed
\end{lemma}

\noindent
We are now ready to prove the following general criterion, which is of independent interest. Here, given a group $G$ acting on a (mediangle) graph $X$, a \emph{gated-cocompact} subgroup refers to a subgroup $H \leq G$ for which there exists some gated subgraph $Y \subset X$ on which $H$ acts cocompactly.

\begin{lemma}\label{lem:GatedCocompact}
Let $\Pi:= \Pi(\Gamma, \mathcal{G},\lambda)$ be a periagroup. Gated-cocompact subgroups of $\Pi$ are virtual retracts. Consequently, if $\Pi$ is residually finite, then gated-cocompact subgroups are separable.
\end{lemma}

\begin{proof}
Let $H \leq \Pi$ be a gated-cocompact subgroup. Fix a gated subgraph $Y \subset M$ on which $H$ acts cocompactly. Let $\mathcal{J}$ denote the set of the hyperplanes tangent to $Y$ (i.e.\ containing edges with a single endpoint in $Y$). Set $\mathrm{Rot}:= \langle \mathrm{stab}_\circlearrowright (J), J \in \mathcal{J} \rangle$. According to \cite[Theorem~6.1]{Mediangle} and its proof (which apply thanks to Lemma~\ref{lem:RotativeStabPeria}), the subgroup $H^+:= \langle H, \mathrm{Rot} \rangle$ splits as a semi-direct product $\mathrm{Rot} \rtimes H$ and $\mathrm{Rot}$ acts on $M$ with $Y$ as a fundamental domain. The semi-direct product decomposition of $H^+$ immediately implies that $H$ is a retract in $H^+$. Next, since $H$ acts cocompactly on $Y$, necessarily $H^+$ acts cocompactly on $M$, which implies that $H^+$ has finite index in $\Pi$. Thus, we have proved that $H$ is a retract in the finite-index subgroup $H^+$. 
\end{proof}

\noindent
Proposition~\ref{prop:ParabolicSeparable} now easily follows from the observation, made in \cite[Corollary~6.6]{Mediangle}, that standard parabolic subgroups are gated-cocompact.

\begin{lemma}\label{lem:ParabolicConvex}
Let $\Pi:= \Pi(\Gamma,\mathcal{G},\lambda)$ be a periagroup. Let $M:= M(\Gamma, \mathcal{G},\lambda)$ its mediangle graph. For every $\Xi \subset V(\Gamma)$, the subgraph $\langle \Xi \rangle \subset M$ is gated. \qed
\end{lemma}

\begin{proof}[Proof of Proposition~\ref{prop:ParabolicSeparable}.]
It follows from Lemma~\ref{lem:ParabolicConvex} that standard parabolic subgroups are gated-cocompact, so our proposition follows from Lemma~\ref{lem:GatedCocompact}.
\end{proof}

\noindent
The second step towards the proof of Theorem~\ref{thm:DoubleCoset} is the following observation:

\begin{lemma}\label{lem:CrossSeparable}
Let $\Pi:= (\Gamma,\mathcal{G},\lambda)$ be a periagroup and let $M:= M(\Gamma, \mathcal{G},\lambda)$ denote its mediangle graph. Assume that $\Gamma$ is finite and that $\Pi$ is residually finite. For all hyperplanes $J,H$ of $M$,
$$\mathrm{Cross}(J,H):= \{ g \in \Pi \mid gJ \text{ transverse to } H\}$$
is separable in $\Pi$. 
\end{lemma}

\begin{proof}
Let $g \in \Pi$ be an element. Notice that, if $gJ$ and $H$ are not transverse, then a simple ping-pong argument shows that 
$$\langle \mathrm{stab}_\circlearrowright(J), \mathrm{stab}_\circlearrowright(H) \rangle = \mathrm{stab}_\circlearrowright(J) \ast \mathrm{stab}_\circlearrowright(H).$$
Otherwise, if $gJ$ and $H$ are transverse, then we can find non-trivial relations between their rotative-stabilisers. Fix two non-trivial elements $r \in \mathrm{stab}_\circlearrowright(J)$ and $s \in \mathrm{stab}_\circlearrowright(H)$. Also, fix a convex $2n$-cycle $C$ in which $gJ$ and $H$ meet. If $n>2$, then $grg^{-1}$ and $s$ must be two reflections in a standard parabolic subgroup (corresponding to the cycle $C$) isomorphic to a dihedral group of order $2n$. In particular, $(grg^{-1}s)^n=1$. And, if $n=2$, then $grg^{-1}$ and $s$ must commute. Therefore, if $N$ denotes a common multiple of all the labels of the edges of $\Gamma$, then either $(grg^{-1}s)^N=1$ or $[r,s]=1$. 

\medskip \noindent
From the previous observations, it follows that $\mathrm{Cross}(J,H)= \alpha^{-1}(1) \cup \beta^{-1}(1)$ where
$$\alpha : \left\{ \begin{array}{ccc} G& \to & G \\ g & \mapsto & (grg^{-1}s)^N \end{array} \right. \text{ and } \beta : \left\{ \begin{array}{ccc} G & \to & G \\ g & \mapsto & [grg^{-1},s] \end{array} \right..$$
These two maps are continuous with respect to the profinite topology; and, since $\Pi$ is residually finite, $\{1\}$ must be closed. Thus, as the union of two closed subspaces $\alpha^{-1}(1)$ and $\beta^{-1}(1)$, we conclude that $\mathrm{Cross}(J,H)$ is separable in $\Pi$, as desired.
\end{proof}

\begin{proof}[Proof of Theorem~\ref{thm:DoubleCoset}.]
If $\Phi = \Psi$, the desired conclusion follows from Proposition~\ref{prop:ParabolicSeparable}. From now on, we assume that $\Phi \neq \Psi$. Let $(\Gamma^+,\lambda^+)$ denote the labelled graph obtained from $(\Gamma,\lambda)$ by adding two new vertices, $u_\Phi$ and $u_\Psi$; by connecting $u_\Phi$ (resp.\ $u_\Psi$) to all the vertices in $\Phi$ (resp.\ $\Psi$) with edges labelled by $2$; and by connecting $u_\Phi$ and $u_\Psi$ with an edge labelled by $2$. Let $\mathcal{G}^+$ be the collection of groups indexed by $V(\Gamma^+)$ such that the group indexed by a vertex in $\Gamma$ coincides with the group from $\mathcal{G}$ labelled by the same vertex and such that the the groups indexed by $u_\Phi$ and $u_\Psi$ are cyclic of order two. Let $\Pi^+$ denote the periagroup $\Pi(\Gamma^+, \mathcal{G}^+,\lambda^+)$. 

\begin{claim}\label{claim:CrossDoubleCoset}
Let $J_\Phi$ (resp.\ $J_\Psi$) denote the hyperplane of $M^+:=M(\Gamma^+,\mathcal{G}^+,\lambda^+)$ containing the edge $[1,u_\Phi]$ (resp.\ $[1,u_\Psi]$). In $\Pi^+$, $\mathrm{Cross}(J_\Psi,J_\Phi)= \langle \Phi \rangle \langle \Psi \rangle \langle u_\Phi,u_\Psi \rangle$.
\end{claim}

\noindent
Let $g \in \mathrm{Cross}(J_\Psi, J_\Phi)$. We know from Lemma~\ref{lem:CarrierRightHypPeriagroup} that $N(J_\Phi)= \langle \Phi,u_\Phi \rangle$ and $N(J_\Psi)= \langle \Psi,u_\Psi \rangle$. Because $gJ_\Psi$ and $J_\Phi$ are transverse, and since carriers of hyperplanes are connected, there must exist a path connecting $1 \in N(J_\Phi)$ to $g \in N(gJ_\Psi)$ that stays in $N(J_\Phi)$ first and then stays in $N(gJ_\Psi)$. This implies that
$$g \in \langle \Phi,u_\Phi \rangle \langle \Psi,u_\Psi \rangle = \langle \Phi \rangle \langle \Psi \rangle \langle u_\Phi,u_\Psi \rangle,$$
as desired. Conversely, for all $a \in \langle \Phi \rangle$, $b \in \langle \Psi \rangle$, and $c \in \langle u_\Phi,u_\Psi \rangle$, we know that $J_\Phi= a^{-1} J_\Phi$ is transverse to $J_\Psi = bc J_\Psi$, so $abcJ_\Psi$ must be transverse $J_\Phi$, i.e.\ $abc \in \mathrm{Cross}(J_\Psi,J_\Phi)$. This concludes the proof of Claim~\ref{claim:CrossDoubleCoset}. 

\medskip \noindent
We know from Lemma~\ref{lem:CrossSeparable} that $\mathrm{Cross}(J_\Psi,J_\Phi)$ is separable in $\Pi^+$. Since $\Pi$ is a standard parabolic subgroup of $\Pi^+$, we also know from Proposition~\ref{prop:ParabolicSeparable} that $\Pi$ is separable in $\Pi^+$, so 
$$\mathrm{Cross}(J_\Psi,J_\Phi) \cap \Pi =  \langle \Phi \rangle \langle \Psi \rangle \langle u_\Phi,u_\Psi \rangle \cap \Pi =  \langle \Phi \rangle \langle \Psi \rangle$$
must be separable in $\Pi$, concluding the proof of our theorem. 
\end{proof}

\section{Conspicial actions of Coxeter groups}

\noindent
It is proved in \cite{MR2646113} that the actions of Coxeter groups on (the median graphs dual to) their mediangle Cayley graphs are conspicial, due to good separability properties of their hyperplane-stabilisers. In order to prove Theorem~\ref{thm:BigIntro}, we need to strengthen this observation. We introduce a notion of \emph{relative conspiciality} and we prove that the actions just mentioned are conspicial relative to a bigger family of subgraphs. First, we need to introduce some vocabulary. 

\begin{definition}
Let $X$ be a mediangle graph and $Y,Z \subset X$ two subgraphs.
\begin{itemize}
	\item $Y$ and $Z$ are \emph{transverse} if they are distinct and if there exists a convex even cycle with two distinct pairs of opposite edges respectively in $Y$ and $Z$.
	\item $Y$ and $Z$ are \emph{tangent} if they are distinct, not transverse, and not separated by a third hyperplane.
\end{itemize}
\end{definition}

\noindent
Notice that, when restricted to hyperplanes, we recover the usual definitions of transversality and tangency. Then, we can generalise conspicial actions:

\begin{definition}
Let $G$ be a group acting on a mediangle graph $X$ and let $\mathcal{P}$ be a collection of subgraphs. The action $G \curvearrowright X$ is \emph{conspicial relative to $\mathcal{P}$} if
\begin{itemize}
	\item for all $g \in G$ and $P \in \mathcal{P}$, if $gP \neq P$ then $P$ and $gP$ are separated by some hyperplane;
	\item for all $g \in G$ and $P,Q \in \mathcal{P}$, if $P$ is transverse to $Q$ then $gP$ is not tangent to $Q$.
\end{itemize}
\end{definition}

\noindent
Notice that conspicial actions, as defined in Section~\ref{section:Conspicial}, correspond to conspicial actions relative to the collection of hyperplanes. Our next statement strengthens the main result of \cite{MR2646113}. 

\begin{thm}\label{thm:CoxRelConspicial}
Let $\Gamma$ be a finite labelled graph and $W(\Gamma)$ the Coxeter group it defines. Let $\mathcal{P}$ denote the family of subgraphs in $M(\Gamma)$ given by the hyperplanes and by the cosets of the standard parabolic subgroups. The action $W(\Gamma) \curvearrowright M(\Gamma)$ is virtually conspicial relative to $\mathcal{P}$. 
\end{thm}

\noindent
The theorem is based on the following criterion, which is a natural analogue of \cite[Theorem~4.1]{MR2646113} and whose proof can be rewritten almost word for word. We write a complete argument for the reader's convenience. In our statement, we use the notations
$$\mathrm{Cross}(P,Q):= \{ g \in G \mid gP \text{ transverse to } Q \}$$
and 
$$\mathrm{Osc}(P,Q) := \{ g \in G \mid gP \text{ tangent to } Q\}$$
for all subgraphs $P$ and $Q$. 

\begin{prop}\label{prop:CriterionConspicial}
Let $G$ be a group acting on a mediangle graph $X$ and $\mathcal{P}$ a $G$-invariant collection of subgraphs. Assume that:
\begin{itemize}
	\item there are finitely many $G$-orbits in $\mathcal{P}$;
	\item for every $P \in \mathcal{P}$, there are finitely many $\mathrm{stab}(P)$-orbits of $Q \in \mathcal{P}$ transverse or tangent to $P$;
	\item for all transverse $P,Q \in \mathcal{P}$, the profinite closure of $\mathrm{stab}(Q) \mathrm{stab}(P)$ is disjoint from $\mathrm{Osc}(P,Q)$;
	\item for every $P \in \mathcal{P}$, the profinite closure of $\mathrm{stab}(P)$ is disjoint from $\mathrm{Cross}(P,P) \cup \mathrm{Osc}(P,P)$.
\end{itemize}
Then $G$ contains a finite-index subgroup whose action on $X$ is conspicial relative to $\mathcal{P}$. 
\end{prop}

\begin{proof}
Let $P_1, \ldots, P_n \in \mathcal{P}$ be representatives modulo the action of $G$. We start by proving that $G$ contains a finite-index $A$ such that $gP$ and $P$ are never transverse for $P \in \mathcal{P}$ and $g \in A$.

\medskip \noindent
First, given an index $1 \leq i \leq n$, we construct a normal finite-index subgroup $A_i \lhd G$ as follows. Notice that there exists a finite set $J \subset G$ such that $\mathrm{Cross}(P_i,P_i)= \mathrm{stab}(P_i)J \mathrm{stab}(P_i)$. This follows from the fact that there are only finitely many $\mathrm{stab}(P_i)$-orbits of elements of $\mathcal{P}$ transverse to $P_i$. By assumption, the closure $F_i$ of $\mathrm{stab}(P_i)$ with respect to the profinite topology is disjoint from $\mathrm{Cross}(P_i)$, and a fortiori from $J$. Thus, for every $j \in J$, there exists a finite-index normal subgroup $A_j \lhd G$ such that $jA_j \cap F_i = \emptyset$. Setting $A_i:= \bigcap_{j \in J} A_j$, we get finite-index normal subgroup of $G$ satisfying $JA_i \cap F_i = \emptyset$.

\begin{claim}\label{claim:AiVScross}
$A_i \cap \mathrm{Cross}(P_i,P_i) = \emptyset$.
\end{claim}

\noindent
Assume that $A_i$ intersects $\mathrm{Cross}(P_i,P_i)= \mathrm{stab}(P_i) J \mathrm{stab}(P_i)$. Then, there exist $j \in J$ and $p,q \in \mathrm{stab}(P_i)$ such that $pjq \in A_i$. Because $A_i$ is normal, this amounts to saying that $qpj \in A_i$. Or equivalently, $p^{-1}q^{-1} \in jA_i$. It follows that $\mathrm{stab}(P_i)$, and a fortiori its closure $F_i$, intersects $JA_i$. This contradicts the definition of $A_i$, and concludes the proof of Claim~\ref{claim:AiVScross}.

\medskip \noindent
Now, set $A:= A_1 \cap \cdots \cap A_n$. Given a $P \in \mathcal{P}$, there must exist $1 \leq i \leq n$ and $h \in G$ such that $P=hP_i$. Then,
$$\mathrm{Cross}(P,P)= \mathrm{Cross}(hP_i,hP_i)= h \mathrm{Cross}(P_i,P_i) h^{-1}$$
must be disjoint from $A_i$, and a fortiori from $A$, as a consequence of Claim~\ref{claim:AiVScross} (and because $A_i$ is normal). 

\medskip \noindent
Thus, we have found a finite-index subgroup of $G$ for which $\mathrm{Cross}(P,P)$ (taken with respect to our finite-index subgroup) is empty for every $P \in \mathcal{P}$. The same argument with $\mathrm{Osc}$ replacing $\mathrm{Cross}$ shows that $G$ contains a finite-index subgroup for which $\mathrm{Osc}(P,P)$ (taken with respect to our finite-index subgroup) is empty for every $P \in \mathcal{P}$. We can also show $G$ contains a finite-index subgroup for which $\mathrm{Osc}(P,Q)$ (taken with respect to our finite-index subgroup) is empty for all transverse $P,Q \in \mathcal{P}$. The proof follows the same strategy, but we write a complete argument for the reader's convenience. Taking the intersection of our three finite-index subgroup then concludes the proof of our proposition.

\medskip \noindent
For every $1 \leq i \leq n$, let $Q^i_1, \ldots, Q_{m_i}^i$ be representatives of hyperplanes transverse to $P_i$ modulo the action of $\mathrm{stab}(P_i)$. For all $1 \leq i \leq n$ and $1 \leq j \leq m_i$, we construct a normal finite-index subgroup $A_{ij} \lhd G$ as follows. Because there are only finitely many $\mathrm{stab}(Q_i^j)$-orbits of elements of $\mathcal{P}$ tangent to $Q_i^j$, we can write $\mathrm{Osc}(P_i,Q_i^j)= \mathrm{stab}(Q_i^j) J \mathrm{stab}(P_i)$ for some finite subset $J \subset G$. By assumption, the closure $F_{ij}$ of $\mathrm{stab}(Q_i^j)\mathrm{stab}(P_i)$ with respect to the profinite topology is disjoint from $\mathrm{Osc}(P_i,Q_i^j)$, and a fortiori from $J$. Thus, as before, we can find a normal finite-index subgroup $A_{ij} \lhd G$ such that $JA_{ij} \cap F_{ij}= \emptyset$. 

\begin{claim}\label{claim:AijVSosc}
$A_{ij} \cap \mathrm{Osc}(P_i,Q_i^j)= \emptyset$.
\end{claim}

\noindent
Assume that $A_{ij}$ intersects $\mathrm{Osc}(P_i,Q_i^j)= \mathrm{stab}(Q_i^j) J \mathrm{stab}(P_i)$. Then, there exist $j \in J$, $q \in \mathrm{stab}(Q_i^j)$, and $p \in \mathrm{stab}(P_i)$ such that $qjp \in A_{ij}$. Since $A_{ij}$ is normal, this amounts to saying that $pqj \in A_{ij}$. Equivalently, $q^{-1}p^{-1} \in jA_{ij}$. This implies that $\mathrm{stab}(Q_i^j) \mathrm{stab}(P_i)$, and a fortiori its closure $F_{ij}$, intersects $JA_{ij}$. This contradicts the definition of $A_{ij}$, and concludes the proof of Claim~\ref{claim:AijVSosc}. 

\medskip \noindent
Now, we set
$$A:= \bigcap\limits_{1 \leq i \leq n} \bigcap\limits_{1 \leq j \leq m_i} A_{ij}.$$
Let $P,Q \in \mathcal{P}$ be two transverse subgraphs. There exist $1 \leq i \leq n$ and $h_1 \in G$ such that $P=h_1P_i$. Since $h_1^{-1}Q$ is transverse to $h_1^{-1}P=P_i$, there exist $1 \leq j \leq m_i$ and $h_2 \in \mathrm{stab}(P_i)$ such that $h_1^{-1}Q=h_2Q_i^j$. Then
$$\begin{array}{lcl} \mathrm{Osc}(P,Q) & = & \mathrm{Osc}(h_1P_i,Q)= h_1 \mathrm{Osc}(P_i,h_1^{-1}Q) h_1^{-1}=  h_1 \mathrm{Osc}(P_i,h_2Q_i^j ) h_1^{-1} \\ \\ & = &  h_1 \mathrm{Osc}(h_2P_i,h_2Q_i^j ) h_1^{-1} = h_1h_2 \mathrm{Osc}(P_i,Q_i^j) h_2^{-1}h_1^{-1} \end{array}$$
must be disjoint from $A_{ij}$, and a fortiori from $A$, as a consequence of Claim~\ref{claim:AijVSosc} (and because $A_{ij}$ is normal). 
\end{proof}

\noindent
Then, Theorem~\ref{thm:CoxRelConspicial} will be essentially a direct consequence of Proposition~\ref{prop:CriterionConspicial} combined with the next observation.

\begin{lemma}\label{lem:SeparabilityCoxeter}
Let $\Gamma$ be a finite labelled graph and $W(\Gamma)$ the Coxeter group it defines. Let $\mathcal{P}$ denote the family of subgraphs in $M(\Gamma)$ given by the hyperplanes and by the cosets of the standard parabolic subgroups. Then:
\begin{itemize}
	\item for every $P \in \mathcal{P}$, $\mathrm{stab}(P)$ is separable;
	\item for all $P,Q \in \mathcal{P}$ either equal, tangent, or transverse, $\mathrm{Cross}(P,Q)$ is closed in the profinite topology. 
\end{itemize}
\end{lemma}

\begin{proof}
Let $P \in \mathcal{P}$ be a subgraph. First, assume that $P$ is a hyperplane. It is proved in \cite{MR2646113} that $\mathrm{stab}(P)$ is separable. Indeed, if $r$ denotes the reflection associated to $P$, then an element $g \in W(\Gamma)$ stabilises $g$ if and only $grg^{-1}=r$. Therefore, if $\phi : W(\Gamma) \to W(\Gamma)$ denotes the maps $g \mapsto grg^{-1}$, then $\phi^{-1}(r) = \mathrm{stab}(P)$. Since $\phi$ is clearly continuous with respect to the profinite-topology, and because the profinite topology is Hausdorff (as $W(\Gamma)$ is residually finite), it follows that $\mathrm{stab}(P)$ is separable. Next, if $P$ is a coset of a standard parabolic subgroup, the desired conclusion is given by Proposition~\ref{prop:ParabolicSeparable}.

\medskip \noindent
Let $P,Q \in \mathcal{P}$ be two subgraphs. First, assume that $P$ and $Q$ are two hyperplanes. Then it is proved in \cite{MR2646113} that $\mathrm{Cross}(P,Q)$ is closed in the profinite topology. Indeed,  if $r$ (resp.\ $s$) denotes the reflection associated to $P$ (resp.\ $Q$), then, given an element $g \in W(\Gamma)$, either $gP$ is transverse to $Q$, and $\langle grg^{-1},s \rangle$ is a finite dihedral group; or $gP$ is not transverse to $Q$, and $\langle grg^{-1},s \rangle$ is an infinite dihedral group. Therefore, $gP$ is transverse to $Q$ if and only if $grg^{-1}s$ has finite order. Fixing an large integer $N$ divisible by every finite order of an element of $W(\Gamma)$, it follows that $\mathrm{Cross}(P,Q)= \phi^{-1}(1)$ where $\phi : W(\Gamma) \to W(\Gamma)$ is defined by $g \mapsto (grg^{-1}s)^N$. Since $\phi$ is clearly continuous with respect to the profinite-topology, and because the profinite topology is Hausdorff (as $W(\Gamma)$ is residually finite), it follows that $\mathrm{Cross}(P,Q)$ is closed in the profinite topology. 

\medskip \noindent
A similar argument applies if, among $P$ and $Q$, one is a hyperplane (say $P$) and the other a coset of standard parabolic subgroup (say $Q$). Let $r$ denote the reflection associated to $P$. Without loss of generality, we assume for simplicity that $Q$ is standard parabolic subgroup, say $\langle \Lambda \rangle$ for some subgraph $\Lambda \subset \Gamma$. Given an element $g \in W(\Gamma)$, if $gP$ is not transverse to $Q$, then $gP$ separates $Q$ and $grg^{-1}Q$, so $grg^{-1}$ does not stabilise $Q$ and therefore does not belong to $\langle \Lambda \rangle$. Otherwise, if $gP$ is transverse to $Q$, then $grg^{-1}$ belongs to $\langle \Lambda \rangle$. Indeed, there must be an edge of $Q= \langle \Lambda \rangle$ flipped by $grg^{-1}$. This edge can be written as $[h,h \ell]$ for some $h \in \langle \Lambda \rangle$ and $\ell \in \Lambda$. From $grg^{-1} h = h \ell$, we deduce that $grg^{-1}=h\ell h^{-1} \in \langle \Lambda \rangle$, as desired. Therefore, $\mathrm{Cross}(P,Q)= \varphi^{-1}(\langle \Lambda \rangle)$ where $\varphi : W(\Gamma) \to W(\Gamma)$ is defined by $g \mapsto grg^{-1}$. We already that $\langle \Lambda \rangle$ is closed in the profinite topology, and $\varphi$ is clearly continuous with respect to the profinite topology, so we conclude that $\mathrm{Cross}(P,Q)$ must be closed in the profinite topology.

\medskip \noindent
Finally, assume that $P$ and $Q$ are two cosets of standard parabolic subgroups. If $P=Q$, then $\mathrm{Cross}(P,Q)= \emptyset$ since otherwise we would find a standard parabolic subgroups with two distinct cosets that intersect, which is impossible. If $P \neq Q$, then we also have $\mathrm{Cross}(P,Q)= \emptyset$. Indeed, because $P$ and $Q$ are gated (Lemma~\ref{lem:ParabolicConvex}), the fact that they are tangent implies that they must intersect (Lemma~\ref{lem:GatedDisjoint}). Without loss of generality, assume that $1$ belongs to $P \cap Q$. So there exist subgraphs $\Lambda,\Xi \subset \Gamma$ such that $P= \langle \Lambda \rangle$ and $Q= \langle \Xi \rangle$. If there exists some $g \in W(\Gamma)$ such that $gP$ is transverse to $Q$, then $g\langle \Lambda \rangle \cap \langle \Xi \rangle$ must contain a convex even cycle. By looking at the generators labelling the edges of this cycle, we deduce that $\Lambda \cap \Xi$ has to contain some edge $[a,b] \subset \Gamma$. But, then,  $\langle a,b \rangle$ yields a convex even cycle in $\langle \Lambda \rangle \cap \langle \Xi \rangle$, proving that $P$ and $Q$ are transverse, which contradicts our assumption that they are tangent. 
\end{proof}

\begin{proof}[Proof of Theorem~\ref{thm:CoxRelConspicial}.]
It suffices to verify that the four items from Proposition~\ref{prop:CriterionConspicial} are satisfied by the action of $W(\Gamma)$ on $M(\Gamma)$ and the collection $\mathcal{P}$. The first two items follows from the cocompactness of $W(\Gamma) \curvearrowright M(\Gamma)$. The fourth item follows from the separability given by the first item of Lemma~\ref{lem:SeparabilityCoxeter}. Finally, let $P,Q \in \mathcal{P}$ be two transverse subgraphs. Because $\mathrm{stab}(Q) \mathrm{stab}(P) \subset \mathrm{Cross}(P,Q)$, and since we know from Lemma~\ref{lem:SeparabilityCoxeter} that $\mathrm{Cross}(P,Q)$ is closed with respect to the profinite topology, it follows that $\mathrm{Cross}(P,Q)$ contains the closure of $\mathrm{stab}(Q) \mathrm{stab}(P)$ with respect to the profinite topology. But $\mathrm{Cross}(P,Q)$ and $\mathrm{Osc}(P,Q)$ are clearly disjoint, so we conclude, as desired, that the closure $\mathrm{stab}(Q) \mathrm{stab}(P)$ is disjoint from $\mathrm{Osc}(P,Q)$. 
\end{proof}

\section{Obstruction to conspiciallity}\label{section:Obstruction}

\noindent
We saw in Section~\ref{section:Conspicial} how to construct an embedding into a graph product of groups from an action on a quasi-median graph. Our periagroups, however, only act on mediangle graphs. The trick is to use the hyperplane structure of mediangle graphs in order to endow them with a structure of space with partitions and then to get a quasi-median graph as explained in Section~\ref{section:Popset}. Thus, periagroups turn out to act on quasi-median graphs. Let us describe this construction in more details.

\paragraph{Conspicial actions on mediangle graphs.} Given a mediangle graph $X$, consider the set of partitions $\mathfrak{P}:= \{ \{ \text{sectors delimited by } J\} \mid J \text{ hyperplane} \}$. Then, $(X,\mathfrak{P})$ is a space with partitions. We refer to the quasi-cubulation $\mathrm{QM}(X):= \mathrm{QM}(X,\mathfrak{P})$ as the \emph{quasi-median completion} of $X$. The next proposition records the basic properties of quasi-median completions of mediangle graphs.

\begin{prop}\label{prop:QMcompletion}
Let $X$ be a mediangle graph. The canonical map $\iota : X \to \mathrm{QM}(X)$ from $X$ to its quasi-median completion $\mathrm{QM}(X)$ satisfies the following properties:
\begin{itemize}
	\item $\iota$ is an isometric embedding that sends cliques to cliques;
	\item for all cliques $C_1,C_2 \subset X$, $\iota(C_1)$ and $\iota(C_2)$ belong to the same hyperplane if and only if so do $C_1$ and $C_2$;
	\item $\iota$ induces a bijection from the hyperplanes of $X$ to the hyperplanes of $\mathrm{QM}(X)$ that preserves (non-)transversality and (non-)tangency.
\end{itemize} 
\end{prop}

\begin{proof}
For all vertices $x,y \in X$, we know from Theorem~\ref{thm:Popset} that the distance in $\mathrm{QM}(X)$ between $\iota(x)$ and $\iota(y)$ is the number of hyperplanes of $X$ on which the orientations $\iota(x)$ and $\iota(y)$ differ, which coincides with the number of hyperplanes separating $x$ and $y$ in $X$, or equivalently with the distance between $x$ and $y$ in $X$. Therefore, $\iota$ is an isometric embedding. 

\medskip \noindent
This implies that $\iota$ sends a clique $C$ to a complete graph $\iota(C)$. We need to verify that $\iota(C)$ is maximal. Otherwise, there exists an orientation $\sigma$ adjacent to all the orientations in $\iota(C)$. Fix two vertices $a,b \in C$. Necessarily, $\sigma$ differs from $\iota(a)$ (resp.\ $\iota(b)$) on a single hyperplane of $X$. Because $\iota(a)$ and $\iota(b)$ only differ on the hyperplane $J$ containing $C$, it follows that $\sigma$ differ from $\iota(a)$ and $\iota(b)$ only on $J$. Let $c \in C$ denote the unique vertex of $C$ that belongs to the sector $\sigma(J)$. Then $\iota(c) = \sigma$, contradicting the fact that $\sigma$ is adjacent to all the vertices in $\iota(C)$. This proves that $\iota(C)$ is indeed a clique in $\mathrm{QM}(X)$.

\medskip \noindent
As a consequence of Theorem~\ref{thm:Popset}, $\iota(C_1)$ and $\iota(C_2)$ belong to the same hyperplane if and only if the orientations from $\iota(C_1)$ and $\iota(C_2)$ all differ on a single hyperplane, which amounts to saying that the vertices in $C_1$ and $C_2$ are pairwise separated by a single hyperplane, or equivalently that $C_1$ and $C_2$ belong to the same hyperplane in $X$. This proves the second item of our proposition.

\medskip \noindent
Finally, the third item of our proposition follows immediately from Theorem~\ref{thm:Popset}. 
\end{proof}

\noindent
Since a group acting on a mediangle graph naturally acts on the quasi-median completion, we can deduce a mediangle analogue of Theorem~\ref{thm:ConspicialQM}. First, we naturally generalise the definition of conspicial actions to mediangle graphs:

\begin{definition}
Let $G$ be a group acting faithfully on a mediangle graph $X$. The action is \emph{conspicial} if 
\begin{itemize}
	\item for every hyperplane $J$ and every $g \in G$, $J$ and $gJ$ are neither transverse nor tangent;
	\item for all transverse hyperplanes $J_1,J_2$ and for every $g \in G$, $J_1$ and $gJ_2$ are not tangent;
	\item for every hyperplane $J$, $\mathfrak{S}(J)$ acts freely on $\mathscr{S}(J)$. 
\end{itemize}
\end{definition}

\noindent
Then, Theorem~\ref{thm:ConspicialQM} generalises to:

\begin{cor}\label{cor:MediangleConspicial}
Let $G$ be a group acting conspicially on a mediangle graph~$X$. 
\begin{itemize}
	\item Fix representatives $\{J_i \mid i \in I\}$ of hyperplanes of $X$ modulo the action of $G$.
	\item Let $\Gamma$ denote the graph whose vertex-set is $\{J_i \mid i \in I\}$ and whose edges connect two hyperplanes whenever they have transverse $G$-translates.
	\item For every $i \in I$, let $G_i$ denote the group $\mathfrak{S}(J_i) \oplus K_i$, where $K_i$ is an arbitrary group of cardinality the number of orbits of $\mathfrak{S}(J_i) \curvearrowright \mathscr{S}(J_i)$.
\end{itemize}
Then $G$ embeds as a virtual retract into the graph product $\Gamma \mathcal{G}$ where $\mathcal{G}= \{G_i \mid i \in I\}$.
\end{cor}

\begin{proof}
The action of $G$ on $X$ extends to an action of $G$ on the quasi-median completion $\mathrm{QM}(X)$ of $X$. It follows from Proposition~\ref{prop:QMcompletion} that this action is also conspicial and that Theorem~\ref{thm:ConspicialQM} applies, providing the desired conclusion.
\end{proof}

\noindent
In order to prove Theorem~\ref{thm:BigIntro}, it then suffices to show that periagroups act virtually conspicially on their mediangle graphs. The rest of the section is dedicated to this question.

\paragraph{The case of periagroups.} The action of a periagroup on its mediangle graph is usually not conspicial. Our goal is to show that a periagroup always splits as a semidirect product between a graph product and a Coxeter group and that the obstruction for the action to be conspicial is contained in the Coxeter part. This is Proposition~\ref{prop:Obstruction}. However, before stating and proving this statement, we need to introduce some notation.

\medskip \noindent
Let $\Pi:= \Pi(\Gamma, \mathcal{G},\lambda)$ be a periagroup and let $M:= M(\Gamma, \mathcal{G}, \lambda)$ denote its mediangle graph. A vertex $u \in V(\Gamma)$ is of \emph{type GP} if all the edges containing $u$ have label $2$; otherwise, $u$ is \emph{of type C}. Let $\Psi$ denote the subgraph of $\Gamma$ induced by the vertices of type C. Notice that the subgroup $\langle \Psi \rangle$ of $\Pi$ is naturally isomorphic to the Coxeter group $C(\Psi)$. According to Lemma~\ref{lem:ParabolicConvex}, the subgraph $\Delta:= \langle \Psi \rangle \subset M$ is gated. A hyperplane of $M$ is \emph{$\Delta$-peripheral} if it is tangent to $\Delta$, i.e.\ it contains an edge having exactly one endpoint in $\Delta$. 

\begin{prop}
The periagroup $\Pi$ splits as $\mathrm{Rot} \rtimes C(\Psi)$ where $\mathrm{Rot}$ is a graph product admitting $\{ \mathrm{stab}_\circlearrowright (J) \mid J \text{ $\Delta$-peripheral}\}$ as a basis. Moreover, $\Delta$ is a fundamental domain for the action of $C(\Psi)$ on $M$. 
\end{prop}

\noindent
The proposition follows from \cite[Theorem~6.1]{Mediangle} and its proof. Let
$$\mathrm{Obs}:= \left\{ g \in \Pi ~\left|~ \begin{array}{l} \text{$\exists J,$ $J$ and $gJ$ transverse or tangent; or $\exists J,H$ transverse}\\ \text{(resp.\ tangent), $gJ$ tangent (resp.\ transverse) to $H$} \end{array} \right. \right\}$$
denote the obstruction for the action $\Pi \curvearrowright M$ to be conspicial; and let
$$\mathrm{CoxObs}:= \left\{ g \in C(\Psi) ~\left|~ \begin{array}{l} \text{$\exists P \in \mathcal{P},$ $J$ and $gJ$ transverse or tangent; or $\exists J,H \in \mathcal{P}$}\\ \text{transverse (resp.\ tangent), $gJ$ tangent (resp.\ transverse) to $H$} \end{array} \right. \right\}$$
denote the obstruction for the action $C(\Psi) \curvearrowright \Delta$ to be conspicial relative to the collection $\mathcal{P}:= \{\text{hyperplanes, cosets of standard parabolic subgroups}\}$. The rest of the section is dedicated to the proof of the following assertion:

\begin{prop}\label{prop:Obstruction}
The inclusion $\mathrm{Obs} \subset \mathrm{Rot} \cdot \mathrm{CoxObs}$ holds.
\end{prop}

\noindent
We start by proving two preliminary lemmas.

\begin{lemma}\label{lem:CarrierInDeltaParabolic}
For every $\Delta$-peripheral hyperplane $J$, the equality 
$$N(J) \cap \Delta = g \langle \text{vertices of type C adjacent to $u$} \rangle$$
holds for some $g \in C(\Psi)$ and some vertex $u \in V(\Gamma)$ of type GP.
\end{lemma}

\begin{proof}
Up to translating by an element of $C(\Psi)$, we can assume that $N(J) \cap \Delta$ contains the vertex $1$. Notice that, as a consequence of Lemma~\ref{lem:CarrierRightHypPeriagroup}, there exists a vertex $u \in V(\Gamma)$ of type GP such that $N(J)= \langle \mathrm{star}(u) \rangle$. Hence
$$N(J) \cap \Delta = \langle \mathrm{star}(u) \rangle \cap \langle \Psi \rangle = \langle \mathrm{star}(u) \cap \Delta \rangle = \langle \text{vertices of type C adjacent to $u$} \rangle,$$
as desired.
\end{proof}

\noindent
In the sequel, in view of Lemma~\ref{lem:NoObsGP} and of our distinction between vertices of type GP and of type C in $\Gamma$, we refer to right hyperplanes as hyperplanes \emph{of type GP} and to hyperplanes \emph{of type C} for the other hyperplanes. Notice that $\Delta$-peripheral hyperplanes are of type GP. As explained in Section~\ref{section:Periagroups}, right hyperplanes, or equivalently hyperplanes of type GP, are naturally labelled by vertices of $\Gamma$. It is worth noticing that these labels are of type GP, making our terminology coherent. 

\begin{lemma}\label{lem:InDeltaInM}
Two hyperplanes $J$ and $H$ crossing $\Delta$ are transverse (resp.\ tangent) in $\Delta$ if and only if they are transverse (resp.\ tangent) in $M$
\end{lemma}

\begin{proof}
If $J$ and $H$ are transverse in $\Delta$, they are clearly transverse in $M$. Conversely, assume that $J$ and $H$ are transverse in $M$. So there is a convex even cycle $C$ containing pairs of opposite edges in $J$ and $H$. Notice that every hyperplane crossing $C$ must be of type C. Indeed, this is clear if $C$ has length $>4$. Otherwise, $C$ is crossed by exactly two hyperplanes, which have to be $J$ and $H$. Therefore, if $C$ is not contained in $\Delta$, then we can find a $\Delta$-peripheral hyperplane $K$ separating $\Delta$ from some vertex of $C$, and this hyperplane has to separate $\Delta$ and $C$ entirely. Necessarily, $K$ is transverse to both $J$ and $K$. Given an element $h \in \mathrm{stab}_\circlearrowright(K)$ that sends $C$ in the sector delimited by $K$ containing $\Delta$, it follows from Lemmas~\ref{lem:RotateCloser} and~\ref{lem:RightAngledCommutation} that $d(hC,\Delta)<d(C,\Delta)$ and that $h$ stabilises both $J$ and $H$. Thus, $hC$ is a new convex even cycle crossed by $J$ and $H$ but closer to $\Delta$ than $C$. By iterating the argument as much as possible, we eventually find a cycle contained in $\Delta$ in which $J$ and $H$ meet, proving that $J$ and $H$ are transverse in~$\Delta$.

\medskip \noindent
In order to conclude the proof of our lemma, it remains to verify that $J$ and $H$ are separated by a hyperplane in $\Delta$ if and only if they are separated by a hyperplane in $M$. If they are separated by a hyperplane, say $K$, in $\Delta$, then we know from the previous paragraph that $K$ cannot be transverse to $J$ nor to $H$ in $M$. This implies that $K$ separates $J$ and $H$ in $M$. Conversely, it is clear that, if $J$ and $H$ are separated by a hyperplane in $M$, then they are separated by a hyperplane in $\Delta$. 
\end{proof}

\begin{proof}[Proof of Proposition~\ref{prop:Obstruction}.]
Let $g \in \mathrm{Obs}$ be an element in our obstruction. Our goal is to prove that $g$ belongs to $\mathrm{Rot} \cdot \mathrm{CoxObs}$.

\medskip \noindent
First, assume that there exists a hyperplane $J$ such that $J$ and $gJ$ are either transverse or tangent (or equivalently, that $J$ and $gJ$ are distinct but not separated by a third hyperplane). Up to translating by an element of $\mathrm{Rot}$, we can assume that $J$ crosses $\Delta$. Fix an edge $e \subset \Delta$ that belongs to $J$. If $ge$ is not contained in $\Delta$, then we can find a $\Delta$-peripheral hyperplane $H$ that separates $\Delta$ and $ge$. Given an element $h \in \mathrm{stab}_\circlearrowright(H)$ that sends $ge$ in the sector delimited by $H$ that contains $\Delta$, we know from Lemma~\ref{lem:RotateCloser} that $d(hge,\Delta)<d(ge,\Delta)$. Moreover, since no hyperplane can separate $J$ and $gJ$, necessarily $H$ and $J$ are transverse, so Lemma~\ref{lem:RightAngledCommutation} implies that $hJ=J$. Thus, up to left-multiplying $g$ with an element of $\mathrm{Rot}$, we can assume that $ge$ is contained in $\Delta$. This implies that $g \in C(\Psi)$. Moreover, we know from Lemma~\ref{lem:InDeltaInM} that $J$ and $gJ$ are transverse or tangent in $\Delta$. We conclude that $g \in \mathrm{CoxObs}$.

\medskip \noindent
Next, assume that there exist two transverse hyperplanes $J$ and $H$ such that $gJ$ is tangent to $H$. There are four cases to distinguish, depending on whether $J$ and $H$ are of type GP or C.

\medskip \noindent
Assume that $J$ and $H$ are both of type GP. Then, we deduce from Lemma~\ref{lem:LabelHypGP} that, because $J$ and $H$ are transverse, their labels are adjacent; and, because $gJ$ and $H$ are tangent, their labels are distinct and non-adjacent. But $J$ and $gJ$ have the same label, so this is not possible.

\medskip \noindent
Assume that $J$ and $H$ are both of type C. Up to translating by an element of $\mathrm{Rot}$, we can assume that $J$ and $H$ both cross $\Delta$. Fix an edge $e \subset \Delta$ that belongs to $J$. If $ge$ does not belong to $\Delta$, then we can find a $\Delta$-peripheral hyperplane $K$ that separates $ge$ from $\Delta$. Let $h \in \mathrm{stab}_\circlearrowright(K)$ be an element that sends $ge$ in the sector delimited by $K$ that contains $\Delta$. According to ~\ref{lem:RotateCloser}, $d(hge,\Delta)<d(ge,\Delta)$. Moreover, because there is no hyperplane separating $H$ and $gJ$, necessarily $K$ and $H$ are transverse, hence $hH=H$ according to Lemma~\ref{lem:RightAngledCommutation}. Thus, up to left-multiplying $g$ with an element of $\mathrm{Rot}$, we can assume that $ge \subset \Delta$. This implies that $g \in C(\Psi)$. Moreover, it follows from Lemma~\ref{lem:InDeltaInM} that $J$ and $H$ are transverse in $\Delta$ and that $H$ and $gJ$ are tangent in $\Delta$. We conclude that $g \in \mathrm{CoxObs}$.

\medskip \noindent
Assume that $J$ is of type C and $H$ of type GP. Up to translating by an element of $\mathrm{Rot}$, we assume that $J$ crosses $\Delta$ and that $H$ is $\Delta$-peripheral. By repeating the argument from the previous paragraph, we can assume that $g \in C(\Psi)$ and that $gJ$ crosses $\Delta$. Indeed, fix an edge $e \subset \Delta$ that belongs to $J$. If $ge$ is not contained in $\Delta$, then we can find a $\Delta$-peripheral hyperplane $K$ that separates $ge$ from $\Delta$. Let $h \in \mathrm{stab}_\circlearrowright(K)$ be an element that sends $ge$ in the sector delimited by $K$ that contains $\Delta$. According to ~\ref{lem:RotateCloser}, $d(hge,\Delta)<d(ge,\Delta)$. Moreover, because there is no hyperplane separating $H$ and $gJ$, necessarily $K$ and $H$ are transverse, hence $hH=H$ according to Lemma~\ref{lem:RightAngledCommutation}. Thus, up to left-multiplying $g$ with an element of $\mathrm{Rot}$, we can assume that $ge \subset \Delta$. This implies, as desired, that $g \in C(\Psi)$ and that $gJ$ crosses $\Delta$. Now, consider $P:= N(H) \cap \Delta$. According to Lemma~\ref{lem:CarrierInDeltaParabolic}, this corresponds to a coset of a standard parabolic subgroup of $C(\Psi)$. Moreover, because $J$ and $H$ are transverse, necessarily $J$ is transverse to $P$ in $\Delta$. So, if we show that $gJ$ and $P$ are tangent in $\Delta$, then we can conclude that $g\in \mathrm{CoxObs}$. 

\medskip \noindent
Assume for contradiction that $gJ$ and $P$ are not tangent in $\Delta$, which amounts to saying that there exists a hyperplane $K$ separating $gJ$ and $P$ in $\Delta$. Because there is no hyperplane separating $H$ and $gJ$, necessarily $K$ must be transverse to either $H$ or $gJ$. It follows from Lemma~\ref{lem:InDeltaInM} that it cannot be transverse to $gJ$, so $K$ and $H$ are transverse. This amounts to saying that, given an edge $e \subset \Delta$ contained in $K$, the endpoints of $e$ have distinct projections to $N(H)$. (Recall that $N(H)$ is gated according to Lemma~\ref{lem:RightHypGated}.) But, fixing an arbitrary vertex $o \in P$, every endpoint of $e$ is connected to $o$ by a geodesic passing through its projection on $N(H)$. The convexity of $\Delta$ implies that the whole geodesic, and in particular our projection, belongs to $\Delta$. Therefore, the projection of $e$ on $N(H)$ is an edge contained in $P$, contradicting the fact that $K$ separates $P$ and $gJ$.

\medskip \noindent
Finally, assume that $J$ is of type GP and $H$ of type C. Then $J$ and $H$ are transverse, $g^{-1}H$ is tangent to $J$, so we are back in the previous case, proving that $g^{-1} \in \mathrm{Rot} \cdot \mathrm{CoxObs}$. Therefore,
$$g \in \mathrm{CoxObs}^{-1} \cdot \mathrm{Rot} = \mathrm{Rot} \cdot \mathrm{CoxObs}^{-1}= \mathrm{Rot} \cdot \mathrm{CoxObs}$$
where the first equality follows from the normality of $\mathrm{Rot}$ and where the second equality follows from the elementary observation that $\mathrm{CoxObs}= \mathrm{CoxObs}^{-1}$. Indeed, if $h \in \mathrm{CoxObs}$ then:
\begin{itemize}
	\item either there exists a hyperplane $J$ such that $J$ and $gJ$ are transverse or tangent, in which case $J$ and $g^{-1}J$ are also transverse or tangent;
	\item or there exist two transverse hyperplanes $J$ and $H$ such that $gJ$ is tangent to $H$, in which case $g^{-1}H$ is tangent to $J$.
\end{itemize}
In both cases, we conclude that $h^{-1}$ also belongs to $\mathrm{CoxObs}$. 
\end{proof}

\section{Proof of the main theorem}

\begin{proof}[Proof of Theorem~\ref{thm:BigIntro}.]
Let $\Pi:= \Pi(\Gamma, \mathcal{G},\lambda)$ be a periagroup with $\Gamma$ finite and let $M:= M(\Gamma, \mathcal{G},\lambda)$ denote its mediangle graph. As explained in Section~\ref{section:Obstruction}, $\Pi$ splits as a semidirect product $\mathrm{Rot} \rtimes C(\Psi)$. We know from Theorem~\ref{thm:CoxRelConspicial} that $C(\Psi)$ contains a finite-index subgroup $H \leq C(\Psi)$ such that the action $H \curvearrowright \Delta$ is conspicial relative to $\mathcal{P}$. In other words, $H \cap \mathrm{CoxObs} = \emptyset$. Setting $\Pi^-:= \mathrm{Rot} \cdot H$, we deduce from Proposition~\ref{prop:Obstruction} that
$$\Pi^- \cap \mathrm{Obs} \subset \Pi^- \cap \left( \mathrm{Rot} \cdot \mathrm{CoxObs} \right) = \mathrm{Rot} \cdot \left( H \cap \mathrm{CoxObs} \right)= \emptyset.$$
Therefore, the action $\Pi^- \curvearrowright M$ is conspicial. It follows from Corollary~\ref{cor:MediangleConspicial} that $\Pi^-$ embeds as virtual retracts into some graph product $\Phi \mathcal{H}$ where $\Phi$ is finite. Let us look at more carefully the vertex-groups of this graph product. 

\medskip \noindent
Notice that, for every hyperplane $J$ of type GP, the rotative-stabilisers of $J$ in $\Pi$, is contained in $\mathrm{Rot}$, and a fortiori in $\Pi^-$. Therefore, the rotative-stabiliser of $J$ in $\Pi^-$ permutes freely and transitively on the sectors delimited by $J$. This implies that the vertex-groups of $\Phi \mathcal{H}$ given by hyperplanes of type GP are rotative-stabilisers of hyperplanes of type GP, which are conjugates of vertex-groups of $\Pi$. Otherwise, if $J$ is a hyperplane of type C, then it delimits exactly two sectors (Lemma~\ref{lem:NotRight}), so the vertex-groups of $\Phi \mathcal{H}$ given by hyperplanes of type C are cyclic of order two. Therefore, the groups in $\mathcal{H}$ are isomorphic to groups in $\mathcal{G}$. 
\end{proof}

\addcontentsline{toc}{section}{References}

\bibliographystyle{alpha}
{\footnotesize\bibliography{PeriaVirtGP}}

\Address

%

\end{document}